\documentclass[11pt]{elsarticle}
\usepackage{fancyhdr}
\usepackage{marginnote}

\usepackage{amsmath}
\usepackage{mathtools}
\usepackage{amsfonts}
\usepackage{amsthm}
\usepackage{amssymb}
\usepackage{color}
\usepackage{dsfont}
\usepackage{geometry}
\newgeometry{tmargin=2.8cm, bmargin=3.5cm, lmargin=1.7cm, rmargin=1.7cm}

\usepackage{color}

\newtheorem{theo}{\bf Theorem}[section]
\newtheorem{coro}{\bf Corollary}[section]
\newtheorem{lem}{\bf Lemma}[section]
\newtheorem{rem}{\bf Remark}[section]
\newtheorem{defi}{\bf Definition}[section]
\newtheorem{ex}{\bf Example}[section]

\newtheorem{prop}{\bf Proposition}[section]

\def\R{{\mathbb{R}}}

\def\r{{\mathbb{R}}}
\def\N{{\mathbb{N}}}
\def\rn{{\mathbb{R}^{N}}}

\def\VTM{{V_T^M(\Omega)}}
\def\VTMi{{V_T^{M,\infty}(\Omega)}}
\newcommand{\sg}{{\mathrm{sgn}_0^+}}
\newcommand{\supp}{{\mathrm{supp}}}
\newcommand{\essinf}{{\mathrm{ess\,inf}}}
\newcommand{\va} {\vec{a}}
\newcommand{\tva} {\vec{\tilde{a}}}
\newcommand{\dep} {\delta}
\newcommand{\et} {d}
\newcommand{\tm} {\mu}

\def\rp{{[0,\infty )}}

\def\vt{{\vartheta^{\tau,r}}}
\def\bt{{\beta^{\tau,r}}}
\def\s{{\sigma}}

\def\bn{{\bar{\nabla}}}

\def\Iid{{ I_i^\delta}}
\def\iIid{{\int_{ I_i^\delta}}}
\def\Mijd{{ M_{i,j}^\delta}}

\def\Mss{{(\Mijd )^{**}}}

\def\OT{{{\Omega_T}}}
\def\iOT{{\int_{\OT}}}

\def\iO{{\int_{\Omega}}}
\def\iQd{{\int_{Q_j^\delta\cap\Omega} }}
\def\Qd{{Q_j^\delta}}
\def\tQd{{\widetilde{Q}_j^\delta}}
\def\iQdn{{\int_{Q_j^\delta\cap\{x:\xi_\delta(x)\neq 0\}} }}

\newcommand{\wt}{\widetilde}

\newcommand{\vp}{\varphi}

\newcommand{\dv}{\mathrm{div}}

\begin{document}

\begin{frontmatter}

\title{Parabolic equation in time and space dependent anisotropic {M}usielak-{O}rlicz spaces in absence of {L}avrentiev's phenomenon}

\author[1,2]{Iwona Chlebicka\corref{mycorrespondingauthor}}
\cortext[mycorrespondingauthor]{Corresponding author}
\ead{i.skrzypczak@mimuw.edu.pl}
\author[1]{Piotr Gwiazda}
\ead{p.gwiazda@mimuw.edu.pl}
 \author[2]{Anna Zatorska--Goldstein\fnref{myfootnote}}
  \ead{azator@mimuw.edu.pl}

\fntext[myfootnote]{The research of I.C. is supported by NCN grant no. 2016/23/D/ST1/01072.  The research of P.G. has been supported by the NCN grant  no. 2014/13/B/ST1/03094. The research of A.Z.-G. has been supported by the NCN grant  no. 2012/05/E/ST1/03232.  The work was also partially supported by the Simons - Foundation grant 346300 and the Polish Government MNiSW 2015-2019 matching fund.}

\address[1]{Institute of  Mathematics, Polish Academy of Sciences, ul. \'{S}niadeckich 8, 00-656 Warsaw, Poland
}
\address[2]{Institute of Applied Mathematics and Mechanics,
%Faculty of Mathematics, Informatics and Mechanics,
University of Warsaw, ul. Banacha 2, 02-097 Warsaw, Poland
}

\begin{abstract}
We study  a general nonlinear parabolic equation  on a Lipschitz bounded domain in $\rn$,
\begin{equation*}
\left\{\begin{array}{l l}
\partial_t u-\dv A(t,x,\nabla u)= f(t,x)&\text{in}\ \ \Omega_T,\\
u(t,x)=0 &\ \mathrm{  on} \ (0,T)\times\partial\Omega,\\
u(0,x)=u_0(x)&\text{in}\ \Omega,
\end{array}\right.
\end{equation*}
with $f\in L^\infty(\Omega_T)$ and $u_0\in L^\infty(\Omega)$. The growth of the monotone vector field $A$ is controlled by a generalized fully anisotropic $N$-function $M:[0,T)\times\Omega\times\rn\to\rp$ inhomogeneous in time and space, and under no growth restrictions on the last variable. It results in the need of the integration by parts formula which  has to be formulated in
an advanced way. Existence and uniqueness of solutions are proven when the Musielak-Orlicz space is reflexive OR  in absence of Lavrentiev's phenomenon. To ensure approximation properties of the space we impose natural assumption that the asymptotic behaviour of the modular function is sufficiently balanced. Its instances  are log-H\"older continuity of variable exponent or optimal closeness condition for powers in double phase spaces.

 The noticeable challenge of this paper is cosidering the problem in non-reflexive and inhomogeneous fully anisotropic space that changes along time.

\end{abstract}


\end{frontmatter}

\tableofcontents
%%%%%%%%%%%%%%%%%%%%%%%%%%%%%%%%%%%%%%%%%%%%%%%%%%%%%%%%%%%%%%%%%%%%%%%%%%%%%%%
\section{Introduction}
%%%%%%%%%%%%%%%%%%%%%%%%%%%%%%%%%%%%%%%%%%%%%%%%%%%%%%%%%%%%%%%%%%%%%%%%%%%%%%%

The main result of the paper is existence and uniqueness of weak solutions to a family of general parabolic equations with bounded data, where the leading part of the operator is  controlled by a generalized nonhomogeneous and anisotropic $N$-function $M:[0,T)\times\Omega\times\rn\to\rp$. We stress out that we  \textbf{do not} impose any particular restriction on the growth of $M$ or its conjugate $M^*$ (neither $\Delta_2$, nor $\nabla_2$), apart from it being an $N$-function (i.e. convex with superlinear growth).  That is, we study the existence in particular in the spaces $L\log L$, $L_{\rm exp}$, and those equipped with modular function of irregular growth.  Note that $M\not\in\Delta_2$ can be trapped between two power-type functions, see~\cite{CGZG} or~\cite{BDMS} for distinct constructions. Other examples of~$N$-functions that do not satisfy  $\Delta_2$-condition are
\begin{itemize}
\item  $M(t,x,\xi)=a(t,x)\left( \exp(|\xi|)-1+|\xi|\right)$;
\item  $M(t,x,\xi)= a(t,x)|\xi_1|^{p_1(t,x)}\left(1+|\log|\xi||\right)+\exp(|\xi_2|^{p_2(t,x)})-1$, when $(\xi_1,\xi_2)\in\R^2$ and $p_i:\OT\to[1,\infty]$.
    This is also a model example to imagine what we mean by an anisotropic modular function.
\end{itemize} 

As special cases, we infer existence for the general parabolic equation in the isotropic or fully anisotropic Orlicz setting without growth conditions, as well as in all reflexive Musielak-Orlicz spaces including variable exponent spaces (with $1<<p(t,x)<<\infty$ under no regularity assumptions on $p$), weighted Sobolev (with bounded weights), and double phase spaces (no matter how far the exponents are if only the weight is bounded). Survey~\cite{IC-pocket} provides an overview of Musielak-Orlicz spaces as the setting for differential equations.

\subsubsection*{Musielak-Orlicz spaces}

Musielak-Orlicz spaces have been studied systematically starting from~\cite{Musielak,Sk1,Sk2} and developed in the context of fluid mechanics~\cite{gwiazda-non-newt,gwiazda-tmna,gwiazda2,Aneta}.  For other recent developments of~the framework of the spaces let us refer e.g. to~\cite{hhk,hht,mmos:ap,mmos2013}. These works concentrate, however, mostly on the case when a modular function satisfies $\Delta_2$ and $\nabla_2$ condition, that is in the separable and reflexive spaces. Even in non-reflexive cases growth conditions on $M$ or at least its conjugate $M^*$ are imposed (this entails separability of~$L_{M^*}$, see~\cite{Aneta} and also~\cite{gwiazda-ren-ell,gwiazda-ren-para}).  Significantly more difficulties can be expected, when the modular function has growth far from polynomial.  Recall that in Musielak-Orlicz spaces with modular function of a growth not comparable to a power function, we have no factorization of the norms like in the Bochner spaces (coming from composition of the Lebesgue norm), see~\cite{IC-pocket}.

In our study the principal role is played by the choice of proper topology and by obtaining a relevant approximation theorems. In the reflexive Orlicz spaces ($M,M^*\in\Delta_2$) the strong  topology coincides with the so-called modular one.  Otherwise, in the Orlicz setting without prescribed growth the relevant topology for PDEs is this weaker modular topology. Indeed, in his seminal paper~\cite{Gossez} Gossez proved that weak derivatives in the Orlicz-Sobolev spaces are strong derivatives with respect to the modular topology. On the other hand, it is well known in the inhomogeneous setting of variable exponent space, as well as of the double-phase space, one is equipped with the density of the smooth functions only if the interplay between the behaviour of the modular function with respect to each of the variables is balanced. It is closely related to the so-called Lavrentiev's phenomenon which occurs when the infimum of a variational functional taken over space of smooth functions is strictly larger than the infimum over the (larger) space of functions on which the functional is defined, see~\cite{LM}. The notion of the phenomenon became naturally generalised to~describe the situation, where functions from certain spaces cannot be approximated by regular ones. For examples of functions  which cannot be modularly approximated in the inhomogeneous spaces see~\cite[Example~3.2]{ZV} by Zhikov and~\cite[Theorem~4.1]{min-double-reg1} by Colombo and Mingione. 

 Let us stress that kind of the Meyers-Serrin theorem, saying that weak derivatives are strong ones with respect to the modular topology, in the Musielak-Orlicz spaces holds only in absence of~Lavrentiev's phenomenon and this is the scope we work in.  To deal with this problem, we provide the modular approximation theorem provided the asymptotic behaviour of the modular function is sufficiently balanced. The condition is optimal within some special cases (variable exponent, double phase together with its borderline case).

\subsubsection*{PDEs in generalised Orlicz spaces}

The study of nonlinear boundary value problems in~non-reflexive Orlicz-Sobolev-type setting originated in the works of Donaldson~\cite{Donaldson} and Gossez~\cite{Gossez2,Gossez3,Gossez}. We refer to a very nice survey on elliptic problems~\cite{Mustonen} by Mustonen and Tienari, while the most relevant reference on parabolic ones are~\cite{ElMes,ElMes2} by Elmahi and Meskine. For the recent advances in this direction we refer e.g. to~\cite{BDMS-arma}. The~case of~vector Orlicz spaces with fully anisotropic modular function is considered starting from~\cite{Ci-fully,Ci-sym} and applied in studies on existence and regularity~\cite{Alb-CPDE11,AlBlFe-p,ACCZG,AlCi-aniso,BarCi-aniso}. For symmetrization-free approach in the setting let us refer to~\cite{Gparabolic,le-ex}. Admitting additional space inhomogeneity in PDEs is considered in e.g.~\cite{gwiazda-ren-ell,gwiazda-ren-para,pgisazg1,pgisazg2}, whereas time and space inhomogeneity to our best knowledge can be found only in the isotropic space in~\cite{ASGpara}.

On the other hand, investigations without   structural conditions of $\Delta_2$--type and thus -- considering nonreflexive spaces -- was also done when the modular function was trapped between some power-type functions  usually briefly described as $p,q$-growth. This direction comes from the fundamental papers~\cite{Marc1,Marc2} by Marcellini and despite it is well understood area it is still an active field especially from the point of view of modern calculus of variations, see e.g.~\cite{ELM,HPHPAK,EMM,EMM2,bcm17,min-double-reg1}.

\subsubsection*{Our objectives}

 Let us stress that the main challenge we face is considering the setting changing with time. Namely,  the modular function $M=M(t,x,\nabla u)$ controlling the growth of the operator and defining Musielak-Orlicz space we investigate depends on the position in the time-space domain $\Omega_T \subset \r^{N+1}$. The basic examples of such a problem would be, besides the variable exponent space $L^{p(\cdot,\cdot)}$ with exponent dependent on $(t,x)\in \OT$, the variable exponent double-phase space $L_{a(\cdot,\cdot)}^{p(\cdot,\cdot),q(\cdot,\cdot)}$ with $p,q:\Omega_T\to (1,\infty)$ and somewhere disappearing weight $a(t,x):\Omega_T\to\rp$, or its Orlicz analogues (see Examples~\ref{ex:Ordp} and~\ref{ex:weOr}). However, our approach is not restricted to the polynomial growth and it is still possible to admit inhomogeneity of a general form.

  Moreover, we admitt the modular function to be fully anisotropic, i.e. $M=M(t,x,\nabla u)$ instead of $M=M(t,x,|\nabla u|)$.   The typical example is~$M(x,\nabla u)=\sum_{i=1}^N |u_{x_i}|^{p_i(x)}$ with various variable exponents in distinct directions. A modular function is called fully anisotropic, if it does not admitt decomposition by separation of roles of~coordinates. The two-dimensional example of fully anisotropic function provided in~\cite{Trud-Ex} is\[B(\xi)=|\xi_1-\xi_2|^\alpha+|\xi_1|^\beta\log^\delta(c+|\xi_1|),\qquad \alpha,\beta\geq 1,\] and $\delta\in \R$ if $\beta>1$, or $\delta>0$ if $\beta=1$, with $c>>1$ large enough to ensure convexity. 
  
In the fully anisotropic Musielak-Orlicz setting the choice of proper functional setting is not obvious. When gradient is considered in the anisotropic space, the function itself can be assumed to belong to various different isotropic spaces. In the anisotropic Orlicz case we can use symmetrization techniques to get optimal Sobolev embedding~\cite{Ci-sym}, but in anisotropic and inhomogeneous Musielak-Orlicz spaces there is no such result.  We choose the most intuitive classical Lebesgue's space. Thus, the framework we investigate involves the functional spaces
\[
\begin{split} \VTM  &=\{u\in L^1(0,T;W_0^{1,1}(\Omega)):\ \nabla u\in L_M(\OT;\rn)\},\\
 \VTMi & =\{u\in L^\infty(0,T; L^2(\Omega))\cap L^1(0,T;W^{1,1}_0(\Omega)):\ \nabla u\in L_M(\OT;\rn)\}\\
 & =\VTM \cap L^\infty(0,T; L^1(\Omega)).
\end{split}
\]
The definition of Musielak-Orlicz space $L_M$  generated by the modular $N$-function $M:[0,T]\times\Omega\times\rn\to\R$ is provided in Section \ref{sec:mo spaces}. Again due to anisotropy and for clarity of the reasoning, we refrain from looking for the biggest space where the data can be considered. 

\medskip

We prove that there exists a unique weak solution to the problem with bounded data, i.e. \begin{equation}\label{eq:bound}
\left\{\begin{array}{ll}
\partial_t u -\dv A (t,x,\nabla u )= f & \ \mathrm{ in}\  \OT,\\
u(t,x)=0 &\ \mathrm{  on} \ (0,T)\times\partial\Omega,\\
u (0,\cdot)=u_{0}(\cdot) & \ \mathrm{ in}\  \Omega,
\end{array}\right.
\end{equation} 
where $[0,T]$ is a finite interval, $\Omega$ is a bounded Lipschitz domain in $ \rn$, $N>1$, $f\in L^\infty(\OT)$, $u_0\in L^\infty(\Omega)$. We consider $A$~belonging to an Orlicz class with respect to the last variable. Namely, we assume that function $A:[0,T]\times\Omega\times\rn\to\rn$  satisfies the following conditions.
\begin{enumerate}[($\mathcal{A}$1)]
\item \label{A1} $A$ is a Carath\'eodory's function, i.e. it is measurable w.r. to $(t,x)\in \OT$ and continuous w.r. to $\xi$;
\item \label{A2} \textbf{Growth and coercivity.} There exists an $N$-function  $M:[0,T]\times\Omega\times\rn\to\r$ and a constant $c_A>0$ such that for all $\xi\in\rn$ we have
\[ M(t,x,\xi)\leq A(t,x,\xi)\xi  \qquad\text{and}\qquad c_A M^*(t,x,A(t,x,\xi))\leq  M(t,x,\xi),\]
where $M^*$ is conjugate to $M$ (see Definitions~\ref{def:Nf} and~\ref{def:conj}).
\item \label{A3} \textbf{Weak monotonicity.} For all $\xi,\eta\in\rn$ and $x\in\Omega$ we have
\[(A(t,x,\xi) - A(t,x, \eta)) \cdot (\xi-\eta)\geq 0.\]
\end{enumerate}
By a weak solution to~\eqref{eq:bound} we mean a function $u \in \VTMi$, such that for any $\vp\in C_c^\infty([0,T)\times \Omega)$
\begin{equation}\label{weak-bound}
-\iOT u \partial_t \vp\, dx\, dt- \iO u(0)\vp(0)\, dx+\iOT A (t,x,\nabla u)\cdot \nabla \vp \, dx\, dt= \iOT f\vp\, dx\, dt.\end{equation} 

We recall that in order to capture arbitrary growth conditions we need to exclude the Lavrentiev phenomenon. For this we impose a certain type of balance condition of $M(t,x,\xi)$ capturing interplay between the behaviour of $M$ for large $|\xi|$ and small changes of the time and the space variables. We start with fully anisotropic conditions ($\mathcal{M}$) and ($\mathcal{M}_p$) which take very intuitive form in the isotropic setting, see conditions ($\mathcal{M}^{iso}$) or ($\mathcal{M}_p^{iso}$) in the Theorem~\ref{theo:main0} below. Recall that the instances  are log-H\"older continuity of variable exponent or optimal closeness condition for powers in double phase spaces.

\medskip

In the fully anisotropic case we shall consider the modular functions satisfying a balance condition.

\begin{enumerate}
\item[($\mathcal{M}$)] \label{M}  Suppose that there exists a function  $\Theta :[0,1]^2\to\rp$ nondecreasing with respect to each of the variables, such that
\begin{equation}
\label{ass:M:vp}  \limsup_{\delta\to 0^+} \Theta (\delta,\delta^{-N})<\infty ,\end{equation} which express the relation between $M(t,x,\xi )$ and \begin{equation}
\label{MIQ}M_{I,Q}(\xi ):= {\essinf}_{\substack{t\in  {I}\cap[0,T],\\ x\in Q\cap\Omega}}M(t,x,\xi ).
\end{equation} We assume that there exist $\xi_0\in\rn$ and $\delta_0>0$, such that for every interval $I\subset\r$, such that $|I|<\delta<\delta_0,$ and every cube $Q\subset \rn$ with ${\rm diam}\, Q<4\delta\sqrt{N}$
\begin{equation}
\label{ass:M:reg:IQ}
\frac{M(t,x,\xi )}{(M_{I,Q})^{**}(\xi)}
\leq \Theta  \left(\delta,  |\xi| \right)\quad\text{for a.e. }t\in I,\ \ \text{a.e. } x\in Q\cap\Omega,\ \text{ and all }\ \xi\in\rn:\ |\xi|>|\xi_{0}|,
\end{equation}
where by $(M_{I,Q})^{**}(\xi)=({(M_{I,Q}(\xi))}^*)^* $, we denote the greatest convex minorant of the infimum from~\eqref{MIQ} (coinciding with the~second conjugate cf. Definition~\ref{def:conj}).
\end{enumerate}

\noindent Nonetheless, when the modular function has at least power-type growth, we relax ($\mathcal{M}$) as follows.

\begin{enumerate}[($\mathcal{M}_p$)]
\item \label{Mp} Suppose that for $\xi\in\rn$, such that $|\xi|>|\xi_p|$ \begin{equation}
\label{M>p}
M(t,x, \xi )\geq c_{gr}|\xi|^p \qquad\text{with }\quad  p>1\text{ and }\  c_p>0 
\end{equation}
and  that there exists a function   $\Theta_p :[0,1]^2\to\rp$ nondecreasing with respect to each of the variables, such that~\eqref{ass:M:reg} holds with
\begin{equation} \label{ass:M:reg-p} \limsup_{\delta\rightarrow0^+}
\Theta_p (\delta,\delta^{-\frac{N}{p}})<\infty .
\end{equation}
\end{enumerate}It should be stressed that  ($\mathcal{M}$) (resp.~($\mathcal{M}_p$)) is applied only in~the~proof of~approximation result --- Theorem~\ref{theo:approx-sp}. For the rest of the reasoning, it suffices to ensure modular density of smooth functions and that $M$ is an $N$-function (see Definition~\ref{def:Nf}). 

\medskip

In the previous paper of the authors~\cite{pgisazg2} the corresponding equation is studied in the spaces not depending on the time variable and with merely integrable data. Moreover, the balance condition is of log-H\"older type, i.e. it describes more narrow class of admissible modular functions (in particular not covering the sharp range of parameters in the double phase space).

 Another similar result was obtained via completely different method in~\cite{ASGpara} for the parabolic problem  posed in the isotropic Musielak--Orlicz setting formulated in the language of maximal monotone graph. Further differences concern facts that  we admitt fully anisotropic setting, which implies additional challenges in approximation (excluding Lavrentiev's phenomenon) and we do not restrict ourselves to domains with $C^2$-boundary. Moreover, unlike our study, the balance condition in~\cite{ASGpara} is imposed on both $M$ and $M^*$ and is of log-H\"older type.

\subsubsection*{Main results}

\medskip

Since the conditions in the isotropic setting are easier to interpret and sometimes optimal, we would like to~present first this special case of the existence result below.

\begin{theo}[Isotropic setting] \label{theo:main0}
Suppose $[0,T]$ is a finite interval, $\Omega$ is a bounded Lipschitz domain in $ \rn$, $f\in L^\infty(\OT)$, $u_0\in L^\infty(\Omega)$, function $A$ satisfies assumptions ($\mathcal{A}$\ref{A1})--($\mathcal{A}$\ref{A3}) with  a locally integrable $N$-function~$M:[0,T]\times\Omega\times\r\to\r$. Assume further that at least one of the following assumptions holds:
\begin{enumerate}
\item[($\mathcal{M}^{iso}$)]  (arbitrary growth) there exists a function  $\Theta^{iso}:\rp^2\to\rp$ nondecreasing with respect to each of the variables, such that
    \begin{equation} \label{ass:M:Th:lim:iso}
    \limsup_{\delta\to 0^+} \Theta^{iso} (\delta, \delta^{-N})<\infty,
    \end{equation}
    and
    \begin{equation} \label{ass:M:reg:iso}
     \frac{M(t,x,s)}{M(s,y,s)}
    \leq \Theta^{iso} \left(|t-s|+c_{sp} |x-y | , s\right)
    \end{equation}

\item[($\mathcal{M}^{iso}_p$)] (at least power-type growth)
    there exists a function  $\Theta^{iso}_p:\rp^2\to\rp$, nondecreasing with respect to each of the variables, satisfying
    \begin{equation} \label{ass:M:Th:lim:iso:p}
    \limsup_{\delta\to 0^+} \Theta^{iso} (\delta, \delta^{-N/p})<\infty,
    \end{equation}   such that for all $s>s_p$
    \begin{equation} \left\{\begin{array}{l}M(t,x, s )\geq c_{gr}\,s^p \qquad\text{with }\  p>1\text{ and }\  c_{gr}>0,\\
    \label{ass:M:reg:iso:p}
    \frac{M(t,x,s)}{M(s,y,s)}
    \leq \Theta_p^{iso} \left(|t-s|+c_{sp} |x-y | , s\right)\end{array}\right.
    \end{equation}

\item[($\mathcal{M}^{iso}_{\Delta_2\nabla_2}$)] (reflexive case) $M,M^*\in\Delta_2$.
\end{enumerate}
Then there exists a unique $u \in \VTMi$ solving~\eqref{eq:bound}, that is satisfying~\eqref{weak-bound} with any $\vp\in C_c^\infty([0,T)\times \Omega)$.
\end{theo}

Our most general result reads as follows.

\begin{theo}[Fully anisotropic setting]
\label{theo:bound} Suppose $[0,T]$ is a finite interval, $\Omega$ is a bounded Lipschitz domain in $ \rn$,   $N>1$,  $f\in L^\infty(\OT)$, $u_0\in L^\infty(\Omega)$, function $A$ satisfy assumptions ($\mathcal{A}$\ref{A1})--($\mathcal{A}$\ref{A3}) with  a locally integrable $N$-function~$M:[0,T]\times\Omega\times\rn\to\rp.$ Assume further that at least one of the following assumptions holds:
\begin{enumerate}
\item[(i)] $M$ satisfies the condition $(\mathcal{M})$ or $(\mathcal{M}_p)$;
\item[(ii)] $M,M^*\in\Delta_2$.
\end{enumerate} Then there exists a unique $u \in \VTMi$ solving~\eqref{eq:bound}, that is satisfying~\eqref{weak-bound} with any $\vp\in C_c^\infty([0,T)\times \Omega)$.
\end{theo}

Directly from Theorem~\ref{theo:bound}  we get a great load of examples, when the existence follows.

\begin{coro}[Skipping ($\cal M$) / $(\mathcal{M}_p)$ -- Orlicz case]
In the pure Orlicz case, i.e. when \[M(t,x,\xi)=M(\xi),\] the balance conditions $(\mathcal{M})$ or $(\mathcal{M}_p)$ do  not carry any information and can be skipped. Therefore, as a direct consequence of the above theorem we get existence of weak solutions to parabolic problem~\eqref{eq:bound} in the anisotropic Orlicz space without growth restrictions. 
\end{coro}

\begin{coro}[Skipping ($\cal M$) / $(\mathcal{M}_p)$ -- reflexive case]
Theorem~\ref{theo:bound} provides existence results when $M,M^*\in\Delta_2$, that is e.g. in the following cases.
\begin{itemize}
\item  When  $M=|\xi|^{p}$, with  $1<p <\infty$, in classical Sobolev spaces ($\nabla u\in L^1(0,T;W^{1,p}_0(\Omega)$) for $p$-Laplace problem $\partial_t u -\Delta_p u = f$, as well as for
\[\partial_t u -\dv \big (a(t,x)|\nabla u|^{p-2}\nabla u \big)= f(t,x)\qquad\text{with}\qquad 0<<a<<\infty.\]
\item  When  $M=|\xi|\log^\alpha (1+|\xi|)$ in $L\log^\alpha L,$ $\alpha>0,$ spaces,  for
\[\partial_t u -\dv \Big (a(t,x)\frac{\log^\alpha({\rm e}+|\nabla u|)}{|\nabla u|}\nabla u \Big)= f(t,x)\qquad\text{with}\qquad  0<<a<<\infty.\]
\item When $M=|\xi|^{p(t,x)}$, with  $1<<p <<\infty$ in variable exponent spaces for
\[\partial_t u -\dv \big (a(t,x)|\nabla u|^{p(t,x)-2}\nabla u )= f(t,x)\qquad\text{with}\qquad 0<<a<<\infty,\ 1<<p<<\infty.\]
\item When $M=|\xi|^p+a(t,x)|\xi|^q$, with $1<p,q<\infty$ and $a:\OT\to[0,\infty)$ being a~bounded function possibly touching zero (no matter how irregular it is) in double phase spaces; for
\[\partial_t u -\dv \Big (b(t,x)\big(|\nabla u|^{p-2}\nabla u +a(t,x) (|\nabla u|^{q-2}\nabla u\big)\Big)= f(t,x)\qquad\text{with}\qquad 0<<b<<\infty.\]
\item  When $M(t,x,\xi)=|\xi|^{p(t,x)}+a(t,x)|\xi|^{q(t,x)}$,  with $1<<p(t,x), q(t,x)<<\infty$ and the~function $a\in L^\infty(\OT)$ nonnegative a.e. in $\OT$ in variable exponent double-phase spaces; for
\[\partial_t u -\dv \Big (b(t,x)\big(|\nabla u|^{p(t,x)-2}\nabla u +a(t,x) (|\nabla u|^{q(t,x)-2}\nabla u\big)\Big)= f(t,x)\qquad\text{with}\qquad 0<<b<<\infty.\]
\item  When $M(t,x,\xi)=M_1(\xi)+a(t,x)M_2(\xi)$, where $M_1,M_2$ are (possibly anisotropic) homogeneous $N$-functions, such that $M_1,M_2,M_1^*,M_2^*\in\Delta_2$, and moreover the~function $a\in L^\infty(\OT)$ is nonnegative a.e. in $\OT$ in Orlicz double phase spaces; for
\[\partial_t u -\dv \left(b(t,x)\Big(\frac{M_1(\nabla u)}{|\nabla u|^2}\cdot\nabla u +a(t,x) \frac{M_2(\nabla u)}{|\nabla u|^2}\cdot\nabla u\Big)\right)= f(t,x)\qquad\text{with}\qquad 0<<b<<\infty.\]
\end{itemize}
\end{coro} 

When the growth of $M$ is far from polynomial, the meaning of the balance condition can be illustrated by the following examples.

\begin{ex}[Orlicz double phase space without growth restrictions]\label{ex:Ordp} When $M(t,x,\xi)=M_1(\xi)+a(t,x)M_2(\xi)$, where $M_1,M_2$ are (possibly anisotropic) homogeneous $N$-functions without prescribed growth  such that $M_1(\xi)\leq M_2(\xi)$ for $\xi:$ $|\xi|>|\xi_0|$, and moreover the~function $a:\OT\to\rp$ is bounded and has a modulus of continuity denoted by $\omega_a$, we infer existence and uniqueness for solution to the problem
\[\partial_t u -\dv \left(b(t,x)\Big(\frac{M_1(\nabla u)}{|\nabla u|^2}\cdot\nabla u +a(t,x) \frac{M_2(\nabla u)}{|\nabla u|^2}\cdot\nabla u\Big)\right)= f(t,x)\in L^1(\Omega_T)\qquad\text{with}\qquad 0<<b<<\infty,\] 
provided \[\limsup_{\delta\to 0} \omega_a(\delta)\frac{\overline{M}_2(\delta^{-N})}{\underline{M}_1(\delta^{-N})}<\infty,\]
where $\underline{M}_1(s):=\inf_{\xi:\,|\xi|=s}{M}_1(\xi)$ and $\overline{M}_2(s):=\sup_{\xi:\,|\xi|=s}{M}_2(\xi)$,
or -- when $M_1$ has at least power growth -- provided\[\limsup_{\delta\to 0} \omega_a(\delta)\frac{\overline{M}_2(\delta^{-N/p})}{\underline{M}_1(\delta^{-N/p})}<\infty.\]
\end{ex}

 \begin{ex}[Weighted Orlicz spaces without growth restrictions]\label{ex:weOr}
 \rm If $M$ has a form %\begin{equation} \label{M:spec}
\[M(t,x,\xi)=\sum_{i=1}^jk_i(t,x)M_i(\xi)+M_0(t,x,|\xi|),\quad j\in\N,\]
%\end{equation}
instead of ($\mathcal{M}$) we assume only that $M_0$ satisfies~($\mathcal{M}^{iso}$), all $M_i$ for $i=1,\dots,j$ are $N$-functions and all $k_i$ are positive and satisfy $\frac{k_i(t,x)}{k_i(s,y)}\leq C_i \Theta_i(|t-s|+c_{sp} |x-y | )$ with $C_i>0$ and $\Theta_i:\rp\to\rp$ and $\Theta_i\in L^\infty$ for $i=1,\dots,j$. Then, according to computations in Appendix, we get that $M$ satifies ($\mathcal{M}$)  when we take  \[\Theta(r, s)=\sum_{j=1}^k \Theta_j(r)+\Theta_0(r,s)\qquad\text{with}\qquad \limsup_{\delta\to 0^+}\Theta(\delta,\delta^{-N})<\infty\] In the case of~($\mathcal{M}_p$) we expect $\limsup_{\delta\to 0^+}\Theta(\delta,\delta^{-N/p})<\infty$.
\end{ex}

\begin{rem}
Theorem~\ref{theo:bound} extends earlier results of the authors~\cite[Proposition~3.1]{pgisazg2}, in which an analogue of the equation~\eqref{eq:bound}, for time-independent modular function was considered.  It results from the fact that Theorem~\ref{theo:approx-sp} here extends \cite[Theorem~2.1]{pgisazg2}.  Namely, the balance conditions (${\cal M}$), resp. (${\cal M}_p$), has more general form and the retrieved approximation results hold not only under the log-H\"older condition in the variable exponent spaces, but also within the sharp range of parameters in the closeness condition in the double phase space. We give more details in Section~\ref{sec:absence}. \end{rem}

\subsection*{Methods}

Since the modular function is time-dependent and fully anisotropic, the identification of limits of approximate sequences is non-trivial. We employ the Minty--Browder monotonicity trick (see Lemma~\ref{lem:mon} in Appendix) adapted to the setting which is neither separable, nor reflexive.   The construction of our solutions holds in~the~absence of~Lavrentiev's phenomenon, i.e. when the functions can be approximated by smooth ones. Indeed, we employ the approximation result even to get a priori estimates. Thus, our key tool are results on approximation in modular topology when  asymptotic behaviour of a modular function is sufficiently balanced. Retrieving the known optimal results we exclude the Lavrentiev phenomenon in the variable exponent spaces under log-H\"older continuity assumption and in the double-phase space within the sharp range of parameters.  Section~\ref{sec:absence} is devoted to the absence of Lavrentiev's phenomenon.

In the spaces equipped with the modular function of a general growth, even in the Orlicz case, one is not equipped with Bochner-type factorization of the norm.  Thus, we need to prove the so-called {\it integration-by-parts-formula}.  It is necessary in passing from distributional formulation of~an equation to the particular class of test functions involving the solution itself. 
 We provide it for problems with merely integrable data, which makes it to be of its own interest.  It does not make the proof longer or essentially more complicated and in this form can find deeper application in studies on problems with data below duality.

We use the framework developed in~\cite{pgisazg1,pgisazg2,gwiazda-ren-ell,gwiazda-ren-cor,gwiazda-ren-para}, where elliptic and parabolic problems in Musielak--Orlicz spaces were studied.
Since $M^*\not\in\Delta_2$, the understanding of the dual pairing is not intuitive. Indeed, in the view of (${\cal A}2$) one expects $A(\cdot,\cdot,\nabla(T_k(u)))$ and $\nabla(T_k(u))$ to belong to the dual spaces, which is not true outside the family of doubling modular functions. Relaxing growth condition on the modular function restricts the admissible classical tools, such as the Sobolev embeddings,~the~Rellich-Kondrachov compact embeddings, or~Aubin-Lions Lemma (applied in~\cite{gwiazda-ren-para} to~prove almost everywhere convergence). Let us note that we obtain uniqueness for weakly monotone operator using the comparison principle (Proposition~\ref{prop:comp-princ}).

\section{Analytical framework -- Musielak-Orlicz spaces} \label{sec:mo spaces}

In this section we introduce briefly the necessary information on the background. Some minor lemmas and classical theorems are listed in Appendix.

\medskip

\noindent \textbf{Notation. } We assume $\Omega\subset\rn$ is a bounded Lipschitz domain, $\Omega_T=(0,T)\times\Omega,$ $\Omega_\tau=(0,\tau)\times\Omega.$ If $V\subset\R^K$, $K\in\N$, is a bounded set, then $C_c^\infty(V)$ denotes the class of smooth functions with support compact in $V$.  Let the symmetric truncation $T_k(f)(x)$ be defined as follows\begin{equation}T_k(f)(x)=\left\{\begin{array}{ll}f & |f|\leq k,\\
k\frac{f}{|f|}& |f|\geq k.
\end{array}\right. \label{Tk}
\end{equation}

\begin{defi}[$N$-function]\label{def:Nf} Suppose $\Omega\subset\rn$ is an open bounded set. A~function   $M:[0,T]\times\Omega\times\rn\to\r$ is called an $N$-function if it satisfies the
following conditions:
\begin{enumerate}
\item $ M$ is a Carath\'eodory function (i.e. measurable with respect to $(t,x)\in\OT$ and continuous with respect to the last variable), such that $M(t,x,0) = 0$, $\essinf_{(t,x)\in\OT}M(t,x,\xi)>0$ for $\xi\neq 0$, and $M(t,x,\xi) = M(t,x, -\xi)$ a.e. in $\Omega$,
\item $M(t,x,\xi)$ is a convex function with respect to $\xi$,
\item $\lim_{|\xi|\to 0}\mathrm{ess\,sup}_{(t,x)\in\OT}\frac{M(t,x,\xi)}{|\xi|}=0$,
\item $\lim_{|\xi|\to \infty}\mathrm{ess\,inf}_{(t,x)\in\OT}\frac{M(t,x,\xi)}{|\xi|}=\infty$.
\end{enumerate}
Moreover, we call $M$ a locally integrable $N$-function if additionally for every measurable set $G\subset\OT$ and every $z\in\rn$ it holds that
\begin{equation}
\label{ass:M:int}\int_G M(t,x,z)\,dx\,dt<\infty.
\end{equation}
\end{defi}

\begin{defi}[Complementary function] \label{def:conj}
The complementary~function $M^*$ to a function  $M:[0,T]\times\Omega\times\rn\to\r$ is defined by
\[M^*(t,x,\eta)=\sup_{\xi\in\rn}(\xi\cdot\eta-M(t,x,\xi)),\qquad \eta\in\rn,\ x\in\Omega.\]
\end{defi}

\begin{rem}\label{rem:f*<g*}If $f(x,\xi)\leq g(x,\xi)$, then $g(x,\xi)^*\leq f^*(x,\xi)$.\end{rem}

\begin{lem} If $M$ is an $N$-function and $M^*$ its complementary, we have \begin{itemize}\item
 the Fenchel-Young inequality \begin{equation}
\label{inq:F-Y}|\xi\cdot\eta|\leq M(t,x,\xi)+M^*(t,x,\eta)\qquad \mathrm{for\ all\ }\xi,\eta\in\rn\mathrm{\ and\ a.e.\ }(t,x)\in\OT.
\end{equation} 
\item the generalised H\"older inequality \begin{equation}\label{inq:Holder} \left|\int_{\Omega} \xi\cdot\eta\,dx\right|\leq 2\|\xi\|_{L_M }\|\eta\|_{L_{M^*} }\quad \mathrm{for\ all\ }\ \xi\in L_M(\OT;\rn),\ \eta\in L_{M^*}(\OT;\rn).\end{equation}
\end{itemize}

\end{lem}

\begin{rem}\label{rem:2ndconj} For any function $f:\r^M\to\r$ the second conjugate function $f^{**}$ is convex and $f^{**}(x)\leq f(x)$. In fact,  $f^{**}$ is a convex envelope of $f$, namely it is the biggest convex function smaller or equal to~$f$.
\end{rem}

\begin{defi}\label{def:MOsp} Let $M$ be a locally integrable $N$-function. We deal with the three  Orlicz-Musielak classes of functions.\begin{itemize}
\item[i)]${\cal L}_M(\OT;\rn)$  - the generalised Orlicz-Musielak class is the set of all measurable functions\\ $\xi:\OT\to\rn$ such that
\[\int_\Omega M(t,x,\xi(t,x))\,dx<\infty.\]
\item[ii)]${L}_M(\OT;\rn)$  - the generalised Orlicz-Musielak space is the smallest linear space containing ${\cal L}_M(\Omega;\rn)$, equipped with the Luxemburg norm
\[||\xi||_{L_M}=\inf\left\{\lambda>0:\int_\OT M\left(t,x,\frac{\xi(t,x)}{\lambda}\right)\,dx\leq 1\right\}.\]
\item[iii)] ${E}_M(\OT;\rn)$  - the closure in $L_M$-norm of the set of bounded functions.
\end{itemize}
\end{defi}
Then
\[{E}_M(\OT;\rn)\subset {\cal L}_M(\OT;\rn)\subset { L}_M(\OT;\rn),\]
the space ${E}_M(\OT;\rn)$ is separable and due to~\cite[Theorem~2.6]{Aneta} the following duality holds \[({E}_M(\OT;\rn))^*=L_{M^*}(\OT;\rn).\]

 We say that an $N$-function $M:[0,T]\times\Omega\times\rn\to\r$ satisfies $\Delta_2$ condition close to infinity (denoted $M\in\Delta_2^\infty$) if there exists a constant $c>0$ and nonnegative integrable function $h:\OT\to\r$ such that for a.e. $(t,x)\in\OT$ it holds 
\begin{equation*}
%\label{D2}
 M(t,x,2\xi)\leq cM(t,x,\xi)+h(t,x)\qquad\text{for all}\quad \xi\in\rn:\ \ |\xi|>|\xi_0|.
\end{equation*}
If $M\in\Delta_2^\infty$, then
\[{E}_M(\OT;\rn)= {\cal L}_M(\OT;\rn)= {L}_M(\OT;\rn)\]
and $L_M(\OT;\rn)$ is separable. When both  $M,M^*\in\Delta^\infty_2$  then $L_M(\OT;\rn)$ is reflexive, see~\cite{GMWK,gwiazda-non-newt}. 

 We face the problem \textbf{without} this structure.

We say that $m$ \textbf{grows essentially more rapidly} than $M$ if \[ {\overline{M}(s)}/{\overline{m}(s)}\to 0\quad\text{ as }\quad s\to\infty\quad\text{ for }\quad {%\displaystyle
\overline{M} (s) = \sup_{ x\in\Omega,\, |\xi|=s}
M (x, \xi)}.\]Let us point out that when $\Omega$ has finite measure and $m$ grows essentially more rapidly than $M$, we have
\begin{equation}
\label{LminEM} 
L_m(\Omega_T;\rn)\subset E_M(\Omega_T;\rn).
\end{equation}

\medskip

We apply the following modular Poincar\'{e}-type inequality.
\begin{theo}[Modular Poincar\'e inequality,~\cite{CGZG}]\label{theo:Poincare}
Let $B:\rp\to\rp$ be an arbitrary function satisfying $\Delta_2$-condition and $\Omega\subset\rn$ be a bounded domain,
then there exist $c^1_P,c^2_P>0$ dependent on $\Omega$ and $N$, such that for every $g\in W^{1,1}(\OT)$, such that $\int_\OT B(|\nabla g|)\,dx\,dt<\infty$, we have
\[\int_\OT B(c^1_P|g|)\,dx\,dt\leq c^2_P
\int_\OT B(|\nabla g|)\,dx\,dt.\]
\end{theo}

%%%%%%%%%%%%%%%%%%%%%%%%%%%%%%%%%%%%%%%%%%%%%%%%%%%%%%%%%%%%%%%%%%%%%%%%%%%%%
%%%%%%%%%%%%%%%%%%%%%%%%%%%%%%%%%%%%%%%%%%%%%%%%%%%%%%%%%%%%%%%%%%%%%%%%%%%%%
\section{Absence of Lavrentiev's phenomenon} \label{sec:absence}
 
Recall that the approximation in modular topology is a natural tool here. In the case of~classical Orlicz spaces, this was noticed by Gossez~\cite{Gossez}. In general Musielak--Orlicz spaces weak derivatives are strong ones with respect to the modular topology only in absence of~Lavrentiev's phenomenon. Thus,  approximation results seem to be of their own interest. Naturally, to exclude problems with approximation an interplay between the~behaviour of the modular function  $M$ for large $|\xi|$ and small changes of the time and the space variables has to be balanced. As already mentioned, in the variable exponent case typical assumption  is log-H\"older continuity of the exponent, whereas in the double-phase spaces the exponents $p<q$ should be close to each other (it is the weight function $a$ who dictates the ellipticity rate of~the energy density).  

We say that this approximation properties is excluding Lavrentiev's phenomenon to link  our study with calculus of variations  and to indicate a critical role of the balance condition therein. It is remarkable that this delicate interplay   acts on a global level. Indeed, in the papers of Zhikov~\cite{zhikov9798,Zhikov2011}, in~the~variable exponent case the conditions on the exponent $p(\cdot)$ ensuring the absence of~Lavrentiev's phenomenon and the density of smooth functions via mollification, are the same as those required for the regularity of~minimizers. The same phenomenon extends to the several other energies including double phase, see for instance~\cite{yags,bcm-st,ELM,zhikov9798,Zhikov2011}.

We shall stress that previously considered conditions in general Musielak-Orlicz spaces, see e.g.~\cite{hhk,hht,mmos:ap,mmos2013,pgisazg2,pgisazg1}, are of log-H\"older type. It is~\cite{yags}  where the way to include the sharp range of parameters in double phase spaces is found. We improve this result to anisotropic and parabolic setting. In further parts of the paper we apply this idea to the existence theory in the space and time dependent spaces.

\medskip

Since the equation has a different structure with respect to time and space we provide two distinct approximation results called for brevity `approximation in space' and `approximation in time'.

\subsection{Approximation in space  }

The approximation in space follows the scheme of~\cite{pgisazg2} modified by ideas of~\cite{yags}. Unlike~\cite{ASGpara} we do not require $M^*$ to satisfy the balance condition.

\begin{theo}[Approximation in space]\label{theo:approx-sp}
Let $\Omega$ be a bounded Lipschitz domain in $\rn$ and a locally integrable $N$-function $M$~satisfy condition ($\mathcal{M}$) (resp.~($\mathcal{M}_p$)). Then for any $\vp$ such that $\vp\in \VTMi$  there exists a sequence $\{\vp_\delta\}_{\delta>0}\subset L^\infty (0,T; C_c^\infty(\Omega))$, such that  $ \vp_\delta\to  \vp$ strongly in $L^1(\Omega_T)$ and $\nabla\vp_\delta\xrightarrow[]{M}\nabla \vp$ in $L_M(\Omega_T;\rn)$ when $\delta\to 0$. Moreover, there exists $c=c(\Omega)>0$, such that $\|\vp_\delta\|_{L^\infty(\Omega)}\leq c \|\vp \|_{L^\infty(\Omega)}$.
\end{theo}

To deal with the approximation in space we start with the construction of an approximate sequence based on the convolution, then we provide the uniform estimate on the star-shape domain and we  conclude the proof of Theorem~\ref{theo:approx-sp}.

\medskip

Let $\kappa_\delta=1- {\delta}/{R}$. For a measurable function $\xi:[0,T]\times\rn\to\rn$ with $\mathrm{supp}\,\xi\subset[0,T]\times\Omega$, we define
\begin{equation}
\label{Sdxi}S_\delta(\xi(t,x)) =%\left(1- {\delta}/{R}\right)^{-1}
 \int_\Omega \rho_\delta( x-y)\xi (t, y/\kappa_\delta)dy,
%=  \int_{B(0,\delta)} \rho _\delta(y)\xi (\kappa_\delta (x-y))dy,
\end{equation}
where $ \rho _\delta(x)=\rho (x/\delta)/\delta^N$ is a standard regularizing kernel on $\rn$  (i.e. $\rho \in C^\infty(\rn)$,
$\mathrm{supp}\,\rho \subset\subset B(0, 1)$ and $\iO \rho (x)dx = 1$, $\rho (x) = \rho (-x)$), such that $0\leq \rho\leq 1$. Let us notice that $S_\delta(\xi)\in C_c^\infty(\rn;\rn)$ and that $S_\delta$ preserves the $L^\infty$ norm.

\begin{lem}\label{lem:step2prev} Suppose $M$ is a locally integrable $N$-function satisfying condition ($\mathcal{M}$), resp.~($\mathcal{M}_p$), and $\Omega$ is a star-shaped domain with respect to a ball $B(0,R)$ for some $R>0$. Let $S_\delta$ be given by~\eqref{Sdxi}. Then there exist  constants $C,\delta_1>0$ independent of $\delta$ such that for all $\delta<\delta_1$
\begin{equation}
\label{in:Md<M}
\int_\OT M(t,x,S_\delta \xi(t,x))\, dx\,dt\leq C
\int_\OT M(t,x, \xi(t,x))\, dx\,dt\qquad\forall_{\xi\in L_M(\OT;\rn)}.
\end{equation} 
\end{lem}
\begin{proof} Fix  $\xi\in L_M(\OT;\rn)$ and note that $f\in L^1(\OT)$ and without loss of generality it can be assumed that \begin{equation}
\label{xileq1}
\|\xi\|_{L^\infty(0,T; L^1(\Omega))}\leq \frac{1}{2^N}.
\end{equation}

 For $0 < \delta < R$ it holds that
\[\overline{\kappa_{\delta} \Omega + \delta B(0, 1)} \subset \Omega.\]
Therefore, $S_\delta\xi\in L^\infty (0,T; C_c^\infty(\Omega))$. 

We fix $0<\delta< R/{4} $. Let $\{\Qd\}_{j=1}^{N_\delta}$ and  $\{\Iid\}_{i=1}^{N^T_\delta}$ be families of sets having the following properties. We denote by $\{I_i^\delta\}_{i=1}^{N_\delta^T}$ a finite family  of closed subintervals $I_i^\delta\subset[0,T]$  of a length not greater than $\delta$, such that $I_i^\delta=[t_i^\delta,t_{i+1}^\delta)$ and $[0,T]=\bigcup_{i=1}^{N_\delta^T}I_i^\delta$.  Moreover, we consider a family of $N$-dimensional cubes   covering the set $\Omega$. Namely, a~family $\{\Qd\}_{j=1}^{N_\delta}$ consists of~closed cubes of edge $2\delta$, such that  $\mathrm{int}\Qd\cap\mathrm{int} Q^\delta_i=\emptyset$ for $i \neq j$ and $\Omega\subset\bigcup_{j=1}^{N_\delta}\Qd$. Moreover, for each
cube $\Qd$ we define the cube $\tQd$ centered at the same point and with parallel corresponding edges of length $4\delta$.

According to condition ($\mathcal{M}$), resp.~($\mathcal{M}_p$), the relation between $M(t,x,\xi )$ and \begin{equation}
\label{Mijd}\Mijd(\xi ):= {\rm ess\,inf}_{t\in  {I}_i^\delta,\, x\in \widetilde{Q}_j^\delta\cap\Omega}M(t,x,\xi )
\end{equation} is as follows
\begin{equation}
\label{ass:M:reg}
\frac{M(t,x,\xi )}{\Mss(\xi)}
\leq \Theta  \left(\delta,  |\xi| \right)\qquad\text{for a.e. }(t,x)\in \Iid\times\Qd\quad\text{and all }\xi\in\rn:\ |\xi|>|\xi_{sp}|,
\end{equation}
where by $\Mss(\xi)=({(\Mijd(\xi))}^*)^* $ we denote the greatest convex minorant (coinciding with the~second conjugate).

  Since $M(t,x,\xi_\delta(x))=0$ whenever $\xi_\delta(x)=0$, we have
\begin{equation}\label{M:div-mult}
\begin{split}
\iOT M(t,x,S_\delta(\xi(t,x)))dxdt & =
\sum_{i=1}^{N^T_\delta}\sum_{j=1}^{N_\delta}\int_{\Iid} \iQd M(t,x,S_\delta(\xi(t,x)))dxdt\\
&=
\sum_{i=1}^{N^T_\delta}\sum_{j=1}^{N_\delta}\iIid \iQdn \frac{M(t,x,S_\delta(\xi(t,x)))}{\Mss(S_\delta(\xi(t,x)))}{\Mss(S_\delta(\xi(t,x)))}dxdt.\end{split}
\end{equation}

Our aim is to show now the following uniform bound
\begin{equation}
\label{M/M<c}\frac{M(t,x,S_\delta(\xi(t,x)))}{\Mss(S_\delta(\xi(t,x)))}\leq C
\end{equation}
for  sufficiently small $\delta>0$, $x\in\Qd\cap\Omega,$ $t\in\Iid$, with $C$ independent of $\delta,t,x,i,j$ and $\xi$.

 Let us fix an arbitrary cube and subinterval and take $(t,x)\in \Iid\times \Qd$. For sufficiently small $\delta$, due to~\eqref{ass:M:reg}, we obtain \begin{equation}
\label{M/M<vpxi}\frac{M(t,x,S_\delta(\xi(t,x)))}{\Mss(S_\delta(\xi(t,x)))}
\leq   \Theta(\delta, |S_\delta(\xi(t,x))|).
\end{equation}

To estimate the right--hand side of~\eqref{M/M<vpxi} we recall definition of $S_\delta$ given in~\eqref{Sdxi}. For any $(t,x)\in\Omega_T$   and each $\delta>0$  we have $\rho _\delta(x-y)\leq   1/\delta^{N}.$ Since~\eqref{xileq1}, we observe that
\begin{equation}
\label{vpxidest}\begin{split}|S_\delta \xi(t,x)|&
\leq \frac{ 1}{ \delta^N }  \int_{\Omega }  |\xi (t, y/\kappa_\delta)|\,dy  \leq \frac{\kappa_\delta^N}{\delta^N}\|\xi\|_{L^\infty(0,T; L^1(\Omega))}\leq \delta^{-N}.\end{split}
\end{equation}
In the case of~($\mathcal{M}_p$), we estimate $|S_\delta \xi(t,x)| \leq  {\delta^{-N/p} }$ using the H\"older inequality.

We combine this with~\eqref{M/M<vpxi},~\eqref{vpxidest}, and by recalling~\eqref{ass:M:vp} (resp.~\eqref{ass:M:reg-p}  if~\eqref{M>p}) to get \begin{equation*}
%\label{M/M<bezxi}
\frac{M(t,x,S_\delta(\xi(t,x)))}{ \Mss(S_\delta(\xi(t,x))) } \leq \Theta(\delta,\delta^{-N })<C,\qquad \left(\text{resp. }\
\frac{M(t,x,S_\delta(\xi(t,x)))}{ \Mss(S_\delta(\xi(t,x))) } \leq \Theta(\delta,\delta^{-N/p })<C\right)
\end{equation*}
for all $\delta<\delta_1$ with some $\delta_1>0$.  Thus, we obtain~\eqref{M/M<c}.

Now, starting from~\eqref{M:div-mult}, noting~\eqref{M/M<c}  and the fact that $\Mss(\xi)=0$ if and only if $\xi=0$, we observe \[
\begin{split}
 \iOT M(t,x,S_\delta\xi(t,x))dx dt&=
\sum_{i=1}^{N^T_\delta}\sum_{j=1}^{N_\delta}\iIid \iQdn \frac{M(t,x,S_\delta(\xi(t,x)))}{\Mss(S_\delta(\xi(t,x)))}{\Mss(S_\delta(\xi(t,x)))}dxdt \\
&\leq C
\sum_{i=1}^{N^T_\delta}\sum_{j=1}^{N_\delta}\iIid \iQdn  {\Mss(S_\delta(\xi(t,x)))}dxdt\\
&\leq C
\sum_{i=1}^{N^T_\delta}\sum_{j=1}^{N_\delta}\iIid \iQd \   {{\Mss}\left( \int_{B(0,\delta)} \rho _\delta(y)\xi (t,\kappa_\delta (x-y))dy\right)}\mathds{1}_{\Qd\cap\Omega}(x) dxdt\\
&\leq C
\sum_{i=1}^{N^T_\delta}\sum_{j=1}^{N_\delta}\iIid  \int_\rn  \  {{\Mss}\left( \int_{B(0,\delta)} \rho _\delta(y)\xi (t,\kappa_\delta (x-y))\mathds{1}_{\Qd\cap\Omega}(x)dy\right)} dxdt\\
&\leq C
\sum_{i=1}^{N^T_\delta}\sum_{j=1}^{N_\delta}\iIid  \int_\rn  {{\Mss}\left( \int_{\rn} \rho _\delta(y)\xi (t,\kappa_\delta (x-y))\mathds{1}_{\tQd\cap\Omega}(x-y)dy\right)} dxdt.\end{split}
\]
Note  that by applying the Jensen inequality  the right-hand side above can be estimated by the following quantity
\[ \begin{split}
&   \ C
\sum_{i=1}^{N^T_\delta}\sum_{j=1}^{N_\delta}\iIid \int_\rn \int_{\rn} \rho _\delta(y) {{\Mss}\left( \xi (t,\kappa_\delta (x-y))\mathds{1}_{\tQd\cap\Omega}(x-y) \right)} dy\,dxdt\\
& \qquad\qquad\qquad\leq C \| \rho _\delta\|_{L^1({B(0,\delta);\rn})}
\sum_{i=1}^{N^T_\delta}\sum_{j=1}^{N_\delta}\iIid \int_\rn {{\Mss}\left( \xi (t,\kappa_\delta z)\mathds{1}_{\tQd\cap\Omega}(z) \right)}  dzdt\\
& \qquad\qquad\qquad\leq  C
\sum_{i=1}^{N^T_\delta}\sum_{j=1}^{N_\delta}\iIid  \int_{\tQd\cap\Omega} {{\Mss}\left( \xi (t,\kappa_\delta z) \right)} dzdt,\end{split}
\]
where we applied the Young inequality for convolution, boundedness of $\rho _\delta$, once again  the fact that $\Mss(\xi)=0$ if~and only if~$\xi=0$. Then, by the definition of  $\Mss$ and the fact it is the greatest convex minorant of $\Mijd$  (see~Remark~\ref{rem:2ndconj}), we realize that
\[ \begin{split} C
\sum_{i=1}^{N^T_\delta}\sum_{j=1}^{N_\delta}\iIid  \int_{\tQd\cap\Omega} {{\Mss}\left( \xi (t,\kappa_\delta z) \right)} dzdt&\leq C\sum_{j=1}^{N_\delta} \int_0^T \int_{{\kappa_\delta}\Qd} {M\left(t,x, \xi (t,x) \right)}  dxdt\\
&\leq C\sum_{j=1}^{N_\delta} \int_0^T \int_{{2}\Qd} {M\left(t,x, \xi (t,x) \right)}  dxdt\\
&\leq  C(N)\iOT {M\left(t,x, \xi (t,x) \right)} dxdt.\end{split}
\]
The last inequality above stands for computation of~a~sum taking into account the measure of~repeating parts of~cubes. We get~\eqref{in:Md<M} by summing up the above estimates.\end{proof}

Now we are in the position to prove approximation in space result.

\begin{proof}[Proof of Theorem~\ref{theo:approx-sp}] If $\Omega$ is a bounded Lipschitz domain in~$\rn$, then there exists a finite family of open sets
$\{\Omega_i\}_{i\in I}$ and a finite family of balls $\{ B^i\}_{i\in I}$ such that
$$\Omega=\bigcup\limits_{i\in I}\Omega_i$$
and every set $\Omega_i$ is star-shaped with respect to ball $B^i$ of radius $R_i$ (see e.g.~\cite{Novotny}).
Let us  introduce the partition of unity $\theta_i$ with
$0\le\theta_i\le 1,$ $\theta_i\in C^\infty_0(\Omega_i), $ ${\rm supp} \,\theta_i=\Omega_i,$ $\sum_{i\in I}\theta_i( x)=1$
 for $x\in\Omega$. We define $Q_i:=(0,T)\times \Omega_i$.

\medskip

 Fix arbitrary $\vp \in \VTMi$. We are going to show that there exists a constant $\lambda>0$ such that
\[ \lim_{\delta\to 0^+}\iOT M\left(t,x, \frac{\nabla S_\delta(  \vp )- \nabla \vp}{\lambda}\right) dx dt = 0,\] where $S_\delta$ is defined in~\eqref{Sdxi}.

Since\[ \iOT M\left(t,x, \frac{\nabla S_\delta(   \vp )- \nabla \vp}{\lambda}\right) dx dt \leq\sum_{i\in I} \int_{Q_i} M\left(t,x, \frac{\nabla S_\delta(\vp)- \nabla \vp}{\lambda}\right) dx dt \]
it suffices to prove it to prove convergence to zero of each integral from the right-hand side.

Let us consider a family of measurable sets  $\{ E_n \}_n$  such that $\bigcup_n E_n = \OT$ and a simple vector valued function
$E^n(t,x)=\sum_{j=0}^n \mathds{1}_{E_j}(t,x) \va_{j}(t,x),$
converging modularly to $\nabla \vp $ with $\lambda_3$ (cf.~Definition~\ref{def:convmod}) whose existence is ensured by Lemma~\ref{lem:dens}.  Note that
\[  \nabla S_\delta(   \vp)- \nabla \vp=\left(\sum_{i\in I}\nabla S_\delta(\theta_i \vp)-S_\delta E^n\right) +(S_\delta E^n -E^n)
	+ (E^n - \nabla\vp) .\]

Convexity of $M(t,x, \cdot)$ implies
	\begin{equation*}%\label{IE:aw14}
	\begin{split}
	&\int_{Q_i} M \left(t, x, \frac{ \nabla S_\delta(  \vp)- \nabla \vp }{ \lambda }\right) \,dx dt=\\
	& \leq	\frac{\lambda_1}{\lambda} \int_{Q_i}  M\left(t, x,
	\frac{\nabla S_\delta(  \vp) -S_\delta E^n}{ \lambda_1} \right) \,dx dt	+ \frac{\lambda_2}{\lambda} \int_{Q_i}  M\left( t,x,
	\frac{S_\delta E^n -E^n }{\lambda_2 } \right) \,dx dt\\
	& \quad  + \frac{\lambda_3}{\lambda} \int_{Q_i}  M\left( t,x,  \frac{ E^n - \nabla( \vp)}{\lambda_3} \right) \,dx dt\\
	&=L^{l,n,\delta}_1+L^{l,n,\delta}_2+L^{l,n,\delta}_3,
	\end{split}
	\end{equation*}
where $\lambda= \sum_{i=1}^3\lambda_i$, $\lambda_i>0$.  We have $\lambda_3$ fixed already. Let us take $\lambda_1=\lambda_3$.

\medskip

We note that $ \vp \in \VTMi$  and for each $i\in I$ we have
$\theta_i\cdot \vp\in L^\infty (Q_i)\cap L^\infty(0,T;L^2(\Omega_i))\cap L^1(0,T;W^{1,1}_0(\Omega_i))$ and 
$\nabla (\theta_i \vp)\ =\ \vp \nabla \theta_i + \theta_i \nabla \vp \ \in\ L_M(\OT;\rn).$ Furthermore, $\sum_{i\in I} \nabla (\theta_i \vp)=\nabla \vp$.

	 Let us notice that
\[L_1^{l,n,\delta}=\frac{\lambda_1}{\lambda} \int_{Q_i}  M\left( t,x,   S_\delta\left( \frac{ E^n  - \sum_{i\in I}\nabla  (\theta_i  \vp)  }{ \lambda_1}\right) \right) \,dx dt.\]
Due to Lemma~\ref{lem:step2prev} the family of operators $S_\delta$ is uniformly bounded from $L_M(\Omega;\rn)$ to~$L_M(\Omega;\rn)$ and we can estimate $0\leq L^{l,n,\delta}_1\leq C L^{l,n,\delta}_3.$ Furthermore, Lemma~\ref{lem:dens} implies that $\lim_{n\to\infty} \lim_{\delta\to 0^+}L^{l,n,\delta}_3= 0,$  which entails $\lim_{l\to\infty} \lim_{\delta\to 0^+} L^{l,n,\delta}_1= 0$ as well.

 Let us concentrate on $L^{l,n,\delta}_2$.
 The Jensen  inequality and then the Fubini theorem lead to
	\begin{equation}\label{IE:aw17}
	\begin{split}
	 \frac{\lambda }{\lambda_2} L^{l,n,\delta}_2  &= \int_{Q_i}  M\left(t, x,  \frac{ E^n (t,x)  -S_\delta E^n (t,x) }{\lambda_2} \right) \,dx dt\\
	& = \int_{Q_i} M \left(t, x, \frac{1}{ \lambda_2} \int_{B(0,\delta)} \varrho_\delta(y) \sum_{j=0}^n [  \mathds{1}_{E_j}(t,x) \va_j (t,x)- \mathds{1}_{E_j}(t,\kappa_\delta(x -  y)) \va_j (t,\kappa_\delta(x -  y)) ]\,dy \right) \,dx\,dt
	\\ &
	\leq
	\int_{B(0,\delta)} \varrho_\delta(y)  \left( \int_{Q_i} M \left(t, x, \frac{1}{\lambda_2} \sum_{j=0}^n [   \mathds{1}_{E_n}(t,x) \va_j (t,x) -\mathds{1}_{E_j}(t,\kappa_\delta(x -  y)) \va_j(t,\kappa_\delta(x -  y)) ] \right) \,dx \right) \,dy\,dt.
	\end{split}	\end{equation}
Using the continuity of the shift operator in $L^1$ we observe that poinwisely
	\[  \sum_{j=0}^n [  \mathds{1}_{E_j}(t,x) \va_j (t,x)- \mathds{1}_{E_j}(t,\kappa_\delta(x -  y)) \va_j(t,\kappa_\delta(x -  y))   ] \xrightarrow[]{\dep\to 0} 0.\]
	Moreover, when we fix arbitrary $\lambda_2>0$ we have
	\[\begin{split}
	 &M \left(t, x, \frac{1}{\lambda_2} \sum_{j=0}^n [   \mathds{1}_{E_j}(t,x) \va_j(t,x) - \mathds{1}_{E_j}(t,\kappa_\delta(x -  y)) \va_j (t,\kappa_\delta(x -  y))  ] \right) \leq \sup_{\eta\in\rn:\ | \eta|=1}M \left(t, x, \frac{1}{ \lambda_2} \sum_{j=0}^n| \va_j|  \eta \right)  <\infty
	 \end{split}\]
and
the Lebesgue Dominated Convergence Theorem provides the right-hand side of \eqref{IE:aw17} converges to zero.

Passing to the limit completes the proof of modular convergence of the approximate sequence. The modular convergence of gradients implies their strong $L^1$-convergence and the Poincar\'e inequality gives the claim. \end{proof}

\subsection{Approximation in time  }

This part was not needed in~\cite{pgisazg2}, due to the lack of time-dependence of $M$. Indeed, therein the following approximation result follows directly from the Jensen inequality. Here we need to examine carefully the uniform estimate and convergence.

\begin{prop}[Approximation in time]\label{prop:time-app}
Let $\Omega$ be a bounded Lipschitz domain in $\rn$, a locally integrable $N$-function $M$~satisfy condition ($\mathcal{M}$) (resp.~($\mathcal{M}_p$)), and $\vp\in \VTMi$. Consider linear mappings $\vp\mapsto\vp_\et$ and $ {\vp} \mapsto\wt{\vp}_\et$, given by
\begin{equation}\label{t-vp-d} \vp_{\et}(t,x) := \frac{1}{{\et}}\int_{t}^{t+\et}\vp(\sigma,x)\,d\sigma \qquad \text{and} \qquad
 \wt{\vp}_{\et}(t,x) := \frac{1}{{\et}}\int_{t-\et}^{t}\vp(\sigma,x)\,d\sigma. 
\end{equation}
When $\et\to 0$, then both $\vp_\et\to\vp$ and $\wt{\vp}_\et\to\vp$ converge strongly in $W^{1,1}_{loc}(\Omega)$. Moreover,
 \[ \nabla (\vp_\et)\xrightarrow[\et\to 0]{}\nabla \vp \quad\text{and}\quad \nabla (\wt{\vp}_\et)\xrightarrow[\et\to 0]{}\nabla{\vp}\quad\text{modularly in }L_M(\Omega_T;\rn).\] Furthermore, $\|\vp_\et\|_{L^\infty(\OT)}\leq \|\vp \|_{L^\infty(\OT)}$ and $\|\wt{\vp}_\et\|_{L^\infty(\OT)}\leq \|\vp \|_{L^\infty(\OT)}$.
\end{prop}

As in approximation in space we need to prove the following uniform estimate. Minor modifications lead to the same result for  $ \vp \mapsto {\vp}_\et$.
\begin{lem}\label{lem:step2prevt} Let a locally integrable $N$-function $M$ satisfy assumptions ($\mathcal{M}$), resp.~($\mathcal{M}_p$). Consider the linear mapping $g\mapsto g_{\et}$ be given by~\eqref{t-vp-d}. Then,  there exist  constant  $C>0$ independent of ${\et}$, such that for all  sufficiently small ${\et}>0$ and every $\eta\in\VTMi$ we have
\begin{equation}
\label{in:Md<M-t}
\int_\OT M(t,x,\wt{\eta}_{\et}(t,x))\, dx\,dt\leq C \int_\OT M\left(t,x, \eta(t,x))\right)\, dx\,dt.
\end{equation}
\end{lem}

\begin{proof} Fix arbitrary $\eta\in \VTMi$ and small $d>0$. Within this proof we understand that $M$ is extended by $0$ outside $[0,T]$. We denote ${I}_i^\et=[t^\et_i,t^\et_{i+1})$, for $i=1,\dots,N^T_\et$, with $|{I}_i^\et|<\et$ and
\begin{equation}\label{wtI} 
\wt{I}_i^\et:=[t_i^\et-\et ,t^\et_{i+1})\cap[0,T].
\end{equation}
In the case of $ \vp \mapsto\wt{\vp}_\et$ we should extend the interval to the right.

 Note that 
\[|\wt{I}_i^\et|\leq 2 |{I}_i^\et|<2\et.\]
We consider
\begin{equation}\label{Mimu}
 {M}_{i,{\et}}(x,\xi):=\inf_{t\in  \wt{I}_i^d} M(t,x,\xi)
\end{equation} and its second conjugate $(M_{i,\et})^{**}(x,\xi)  =({(M_{i,\et}(x,\xi))}^*)^* $, see~Remark~\ref{rem:2ndconj}. Recall also notation for infimum over a cylinder~\eqref{Mijd}.

We observe that since for every $i=1,\dots,N^T_\et$ we have
\[\wt{I}^{\et}_i\subset {I}^{2\et}_{\lceil i/2 \rceil},\]
thus for a.e. $x\in\wt{Q}_j^\et$, $j=1,\dots,N_{\et}$,  and $i=1,\dots,N^T_{\et}$ also
$M_{\lceil i/2 \rceil,j}^{2\et}(\xi)\leq M_{i,\et }(x,\xi)$. Therefore,
\begin{equation}\label{mimij}
  \frac{M(t,x, \xi)}{(M_{i,\et })^{**}(x,\xi)}\leq\frac{M(t,x,\xi)}{(M_{\lceil i/2 \rceil,j}^{2\et})^{**}(\xi)}\leq\Theta\left(2\et,|\xi| \right).
\end{equation}

 Since $M(t,x,\xi)=0$ whenever $\xi=0$, we have
\begin{equation}
\label{M:div-mult-t}\begin{split}
 \iOT M(t,x, \wt{\eta}_{\et} (t,x)))\,dx\,dt&= \sum_{i=1}^{N_\et^T}\int_{\Omega} \int_{I_i^\et} M(t,x,  \wt{\eta}_{\et} (t,x))\, dt\, dx=\\&=\sum_{i=1}^{N_\et^T}\int_{\Omega} \int_{I_i^\et}  \frac{M(t,x, \wt{\eta}_{\et} (t,x))}{(M_{i,\et})^{**}(x, \wt{\eta}_\et(t,x))}{(M_{i,\et})^{**}(x,\wt{\eta}_{\et} (t,x))}\,dt\,dx.\end{split}
\end{equation}

We fix any $\xi\in\Omega$ and choose $Q_j^\et$ including $x$. Then, using assumption~\eqref{mimij} and (${\cal M}$), we realize that  for arbitrary $ t \in I_i^\et$ we get
\begin{equation}
\label{M/M<c-t}\frac{M(t,x, \wt{\eta}_{\et} (t,x))}{(M_{i,\et})^{**}(x,\eta_{\et}(t,x))}\leq \frac{M(t,x,\wt{\eta}_{\et}(t,x))}{(M_{\lceil i/2 \rceil,j}^{2\et})^{**}(\eta_{\et} (t,x))}\leq \Theta\left(2\et, |\wt{\eta}_{\et} (t,x)|\right).
\end{equation} We want to estimate it from above by a constant independent of $x,t,i,j$, and ${\et}$. Since without loss of generality it can be assumed that $\|\eta \|_{L^\infty(0,T;L^\infty(\Omega))}\leq 1$, we have
\begin{equation}
\label{xidest}\begin{split}|\wt{\eta}_{\et} (t,x)|&
\leq   \frac{1}{\et}\int_{t-\et}^{t} \left|{\eta(s,x)  }\right|ds \leq  |\Omega|\cdot\|\eta  \|_{L^\infty(0,T;L^\infty(\Omega))}\leq c(\Omega).\end{split}
\end{equation}
Then, we have $\Theta(2\et, | \wt{\eta}_{\et} (t,x)| )\leq \Theta ( 2\et,  c(\Omega) )\leq c$  by assumption  ($\mathcal{M}$), resp.~($\mathcal{M}_p$). Thus, we can estimate the right-hand side of~\eqref{M/M<c-t} by $c$. Using it in~\eqref{M:div-mult-t}, then extending the domain of integration  we obtain
\[
\begin{split}
 \iOT M(t,x,\wt{\eta}_{\et}(t,x))\,dx\,dt &\leq
c\sum_{i=1}^{N_\et^T}\iO \int_{I_i^\et}(M_{i,\et})^{**}
 \left(x,\int_\R\frac{1}{{\et}}\mathds{1}_{[0,\et)}(\sigma)\eta(t-\sigma,x)\,d\sigma
 \right)\,dt\,dx\\
 & =
 c\sum_{i=1}^{N_\et^T}\iO \int_{I_i^\et}(M_{i,\et})^{**}
 \left(x,\int_\R\frac{1}{{\et}}\mathds{1}_{[0,\et)}(\sigma)\mathds{1}_{I_i^\et}(t)\eta(t-\sigma,x)\,d\sigma
 \right)\,dt\,dx\\
 &
\leq
c\sum_{i=1}^{N_\et^T}\iO \int_{I_i^\et}(M_{i,\et})^{**}
 \left(x,\int_\R\frac{1}{{\et}}\mathds{1}_{[0,\et)}(\sigma)\mathds{1}_{\wt{I}_i^\et}(t-\sigma) \eta(t-\sigma,x)\,d\sigma
 \right)\,dt\,dx. 
\end{split}
\]
Finally we apply  the Jensen inequality, the fact that the second conjugate is (the greatest convex) minorant, and the Young inequality for convolution to get
\[
\begin{split}
 \iOT M(t,x,  \wt{\eta}_{\et} )\,dx\,dt 
 &
\leq
c\sum_{i=1}^{N_\et^T}\iO \int_{I_i^\et}\int_\R\frac{1}{{\et}}\mathds{1}_{[0,\et)}(\sigma)(M_{i,\et})^{**}
 \left(x,\mathds{1}_{\wt{I}_i^\et}(t-\sigma)\eta(t-\sigma,x)\right)\,d\sigma \,dt\,dx\\
 &
\leq
c\sum_{i=1}^{N_\et^T}\iO \int_{I_i^\et}\int_\R\frac{1}{{\et}}\mathds{1}_{[0,\et)}(\sigma)M
 \left(t-\sigma,x,\eta(t-\sigma,x)\right)\,d\sigma \,dt\,dx\\
 &
\leq
c \iO \int_{\R}\int_\R\frac{1}{{\et}}\mathds{1}_{[0,\et)}(\sigma)M
 \left(t-\sigma,x,\eta(t-\sigma,x)\right)\,d\sigma \,dt\,dx\\
 &
\leq
c \iO \frac{1}{{ \et}}\|\mathds{1}_{[0,d)}(\cdot)\|_{L^1(\R)}\cdot\|M
 \left(\cdot,x, \eta(\cdot,x)\right)\|_{L^1(\R)}\,dx\\ 
 &\leq C \| M
 \left(\cdot, \cdot,  \eta (\cdot, \cdot) \right)\|_{L^1(\OT)},
\end{split}
\]
what concludes the proof.
\end{proof}

We are in position to prove the approximation in time of regularizations defined in~\eqref{t-vp-d}.
\begin{proof}[Proof of Proposition~\ref{prop:time-app}] We show only the modular convergence $\nabla(\wt{\vp}_{\et})\to \nabla\vp$ using Lemma~\ref{lem:step2prevt}, because precisely the same reasoning works also for $\nabla( {\vp}_{\et})\to \nabla\vp$. We notice that from the definition of this regularization $ \vp_{{\et}} \in W^{1,\infty}(0,T;\VTMi)$ and $\nabla(\wt{\vp} _{\et})=\wt{(\nabla \vp)}_{\et} $. It suffices now to prove the modular convergence
\[ \nabla(\vp_{\et})\xrightarrow[{\et}\to 0]{M}\nabla\vp\quad\text{ in }\quad L_M(\Omega_T;\rn).\]

Let us consider a family of measurable sets  $\{ \tilde{E}_n \}_n$  such that $\bigcup_n \tilde{E}_n = \OT$ and a simple vector valued function $\tilde{E}^n(t,x)=\sum_{j=0}^n \mathds{1}_{\tilde{E}_j}(t,x) \tva_{j}(t,x),$ converging modularly to $\nabla \vp $ with $\tilde{\lambda}_4$ (cf.~Definition~\ref{def:convmod}) which exists due to Lemma~\ref{lem:dens}.    Note that
\[  \nabla (\wt{\vp}_{\et})- \nabla \vp= \left( \nabla (\wt{\vp}_{\et})-(\tilde{E}^n)_{\et}\right) +( (\tilde{E}^n)_{\et} - \tilde{E}^n )
	+ ( \tilde{E}^n  - \nabla\vp).\]

Convexity of $M(t,x, \cdot)$ implies
	\begin{equation*}%\label{IE:aw14}
	\begin{split}
	 \iOT M \left(t, x, \frac{ \nabla (\wt{\vp}_{\et})- \nabla \vp }{ \lambda }\right) \,dx dt& \leq	\frac{\tilde{\lambda}_1}{\tilde{\lambda}} \iOT  M\left(t, x,
	\frac{\nabla(\wt{ \vp}_{\et})-(\tilde{E}^n)_{\et}}{ \tilde{\lambda}_1} \right) \,dx dt\\
	&\quad +\frac{\tilde{\lambda}_2}{\tilde{\lambda}} \iOT M\left( t,x,
	\frac{(\tilde{E}^n)_{\et}-\tilde{E}^n }{\tilde{\lambda}_2} \right) \,dx dt+ \frac{\tilde{\lambda}_3}{\tilde{\lambda}} \iOT  M\left( t,x,  \frac{ \tilde{E}^n - \nabla  \vp }{\tilde{\lambda}_3} \right) \,dx dt\\
	&=L^{ n,\et}_1+L^{ n,\et}_2+L^{n}_3,
	\end{split}
	\end{equation*}
where $\tilde{\lambda}= \sum_{i=1}^3\tilde{\lambda}_i$, $\tilde{\lambda}_i>0$.  We have $\tilde{\lambda}_3$ fixed already. Let us take $\tilde{\lambda}_1=\tilde{\lambda}_3$.

In order to pass to the limit with ${\et}\to 0$, we apply Lemma~\ref{lem:step2prevt}  estimating \[0\leq \lim_{{\et}\to 0}L^{ n,\et}_1\leq C L^{ n }_3.\]
 Furthermore,  Lemma~\ref{lem:dens} implies that $\lim_{n\to\infty}  L^{ n }_3= 0,$  which entails $\lim_{n\to\infty}  \limsup_{{\tm}\to \infty} L^{ n,{\tm}}_1= 0$ as well.

 Let us concentrate on $L^{ n,{\et}}_2$.
The Jensen  inequality and then the Fubini theorem lead to
	\begin{equation}\label{IE:aw17'}
	\begin{split}
	\frac{\tilde{\lambda} }{\tilde{\lambda}_2} L^{ n,{\et}}_2  
	& =\sum_{i=1}^{N_{\et}^T}\int_{\Omega }\int_{I_i^{\et}}
  M \left(t, x, \frac{1}{\tilde{\lambda}_2} \int_{\r} \frac{1}{\et}\mathds{1}_{[0,\et)}(s ) \sum_{j=0}^n [  \mathds{1}_{\tilde{E}_j}(t,x) \tva_j (t,x)- \mathds{1}_{\tilde{E}_j}(s-t,x) \tva_j (s-t,x) ]\,ds \right)\, dt\,dx\\ &
	\leq 
	  \sum_{i=1}^{N_{\et}^T}\int_{\Omega }\int_{I_i^{\et}}  \frac{1}{\et}\mathds{1}_{[0,\et)}(s )
	   M \left(t, x, \frac{1}{\tilde{\lambda}_2} \sum_{j=0}^n [   \mathds{1}_{\tilde{E}_j}(t,x) \tva_j (t,x) -\mathds{1}_{\tilde{E}_j}(s-t,x) \tva_j(s-t,x) ]
	   \right) \,ds\,dt\,dx\\ &
	\leq 
	  \int_{\Omega_i}  \sum_{i=1}^{N_{\et}^T}\int_{I_i^\et}
	   M \left(t, x, \frac{1}{\tilde{\lambda}_2} \sum_{j=0}^n [   \mathds{1}_{\tilde{E}_j}(t,x) \tva_j (t,x) -\mathds{1}_{\tilde{E}_j}(s-t,x) \tva_j(s-t,x) ]
	   \right) \,dt\,dx.
	\end{split}
	\end{equation}
We let ${\et}\to 0$. Notice that using the continuity of the shift operator in $L^1$ we observe that poinwisely
	\[  \sum_{j=0}^n [  \mathds{1}_{\tilde{E}_j}(t,x) \tva_j (t,x)- \mathds{1}_{\tilde{E}_j}(s-t,x) \tva_j(s-t,x)   ]\xrightarrow[{\et}\to  0]{} 0,\]
because $s-t<{\et}$. Moreover, for arbitrary $\tilde{\lambda}_2>0$ we have
	\[\begin{split}
	 &M \left(t, x, \frac{1}{\tilde{\lambda}_2} \sum_{j=0}^n [  \mathds{1}_{\tilde{E}_j}(t,x) \tva_j (t,x)- \mathds{1}_{\tilde{E}_j}(s-t,x) \tva_j(s-t,x)   ] \right) \leq \sup_{\eta\in\rn:\ | \eta|=1}M \left(t, x, \frac{1}{ \tilde{\lambda}_2} \sum_{j=0}^n| \tva_j|  \eta \right)  <\infty
	 \end{split}\]
and
the Lebesgue Dominated Convergence Theorem provides the right-hand side of \eqref{IE:aw17'} converges to zero.

Passing to the limit completes the proof of modular convergence of the approximating sequence. The modular convergence of gradients implies their strong $L^1$-convergence and Poincar\'e inequality ends the proof. The $L^\infty$ norm is preserved directly due to the formula~\eqref{t-vp-d}.
\end{proof}
%%%%%%%%%%%%%%%%%%%%%%%%%%%%%%%%%%%%%%%%%%%%%%%%%%%%%%%%%%%%%%%%%%%%%%%%%%%%%%%%%
%%%%%%%%%%%%%%%%%%%%%%%%%%%%%%%%%%%%%%%%%%%%%%%%%%%%%%%%%%%%%%%%%%%%%%%%%%%%%%%%%
\section{Auxiliary results}
In this section we provide integration-by-parts formula, comparison principle, and existence to a regularized problem.

\subsection{Integration-by-parts formula}\label{sec:ren} 
 
Unless the growth of the modular function is comparable with the power function the Musielak-Orlicz setting inherits from Orlicz spaces the well-known problem with the so-called {\it{integration-by-parts formula}}, see~\cite{ZaGa,IC-pocket}.  As it can be expected, it holds only in absence of Lavrentiev's phenomenon and in spite of its similarity to~\cite[Lemma~2.1]{pgisazg2} its proof substantially involves `approximation in time' result  (Proposition~\ref{prop:time-app}). 

\begin{prop}[Integration-by-parts formula]\label{prop:intbyparts}
Suppose $M$ is a locally integrable $N$-function satisfying condition ($\mathcal{M}$) (resp.~($\mathcal{M}_p$)),  $u:\Omega_T\to\r$ is a measurable function such that for every $k\geq 0$, $T_k(u)\in \VTM$, $u(t,x)\in L^\infty([0,T];L^1(\Omega))$. Let us assume that there exists $u_0\in L^1(\Omega)$ such that $u_0(x):=u(0,x)$. Furthermore, assume that there exist $A\in L_{M^*}(\Omega_T;\rn)$  and $F\in L^1(\Omega_T)$ satisfying
\begin{equation}
\label{eq:prop:int-by-parts-1}
-\int_{\OT}(u-u_0)\partial_t \vp \,dx\,dt+\int_{\OT}A\cdot \nabla\vp \,dx\,dt=\int_{\OT}F\, \vp \,dx\,dt,\qquad \forall_{\vp\in{C_c^\infty}([0;T)\times \Omega)}.
\end{equation}
Then
\begin{equation}
\label{claim:prop:int-by-parts-1}
-\int_{\OT} \left(\int_{u_0}^u h(\s)d\s\right) \partial_t \xi \ \,dx\,dt+\int_{\OT}A\cdot \nabla (h(u)\xi) \,dx\,dt=\int_{\OT}F h(u)\xi \,dx\,dt
\end{equation}
holds for all $h\in W^{1,\infty}(\r)$, such that $\supp (h')$ is compact and all $\xi\in \VTMi$, such that $\partial_t\xi\in L^\infty(\OT)$ and $\supp\xi(\cdot,x)\subset[0,T)$ for a.e. $x\in\Omega$, in particular for  $\xi \in C_c^\infty([0,T)\times\overline{\Omega})$.
\end{prop}

\begin{proof} Let $h\in W^{1,\infty}(\r)$ be such that $\supp (h')$ is compact. Let us note that  $h_1,h_2:\R\to\R$ given by
\[h_1(t)=\int_{-\infty}^t (h')^+(s)\,ds,\ g^+:=\max\{0,g\},\quad h_2(t)=\int_{-\infty}^t (h')^-(s)\,ds,\ g^-:=\min\{0,g\}\]
are   Lipschitz continuous functions. Moreover, $h_1$ is non-decreasing, $h_2$ is non-increasing, $h=h_1+h_2$ and in the case {\it ii)} $h_1(0)=0=h_2(0).$ In both cases there exists $k>0$ such that $\supp(h')\subset [-k,k]$, thus $h(u)=h(T_k(u))=h_1(T_k(u))+h_2(T_k(u))$. Furthermore, \[h_1(T_k(u)),h_2(T_k(u))\in L^\infty(\OT)\quad\text{and}\quad \nabla (h_1(T_k(u))),\nabla (h_2(T_k(u)))\in L_M(\OT;\rn).\] It follows from the existence of modularly converging sequence $\nabla (T_k(u))_\delta$, cf. Theorem~\ref{theo:approx-sp}, which via Definition~\ref{def:convmod} implies uniform integrability of~$\left\{M\left(x,{h_1'((T_k (u))_\delta)\nabla (T_k(u))_\delta}/{\lambda}\right)\right\}_\delta$ for some $\lambda>0$.

We start with the proof for nonnegative $\xi$, which we extend in the following way \begin{equation}\label{xi:ext}
\xi(t,x)=\left\{\begin{array}{ll}0,&t>T,\\
\xi(t,x),&t\in[0,T],\\
\xi(-t,x),& t<0.
\end{array}\right.\end{equation} Additionally, we extend $u(t,x)=u_0(x)$ for $t<0$.  Let us define
\begin{equation} \label{t-zeta} 
 \zeta := h_1(T_k(u))\xi. \end{equation}
We fix ${\et}>0$ and recall
\begin{equation}\label{t-zeta-d} \zeta_{\et}(t,x) := \frac{1}{{\et}}\int_t^{t+{\et}}\zeta(\sigma,x)\,d\sigma \qquad \text{and} \qquad
 \wt{\zeta}_{\et}(t,x) := \frac{1}{{\et}}\int_{t-{\et}}^t\zeta(\sigma,x)\,d\sigma. 
\end{equation}
Note that due to the same reasoning as for $h(u)$, also $\zeta_{\et}, \wt{\zeta}_{\et}(t,x):\OT\to\r$ belong to $\VTMi$. Furthermore, $\partial_t \zeta_{\et},\partial_t \wt{\zeta}_{\et}(t,x)\in L^\infty(\OT)$. Then $\zeta_{\et}(T,x)= \wt{\zeta}_{\et}(0,x)=0$ for all $x\in\Omega$ and ${\et}>0$. We can use approximate sequences $(\zeta_{\et})_\delta, (\wt{\zeta}_{\et})_\delta$ from Theorem~\ref{theo:approx-sp}  as  test functions  in~\eqref{eq:prop:int-by-parts-1}. Indeed, we have  $(\zeta_{\et})_\delta, (\wt{\zeta}_{\et})_\delta\in  W^{1,\infty} (0,T; C_c^\infty(\Omega))$ to which we can extend the functionals described by each of the integrals we arrive, namely in
\begin{equation}\label{1sttestve}
  \iOT A\cdot \nabla ( (\zeta_{\et})_\delta)  \,dx\,dt-\iOT F ( \zeta_{\et})_\delta \,dx\,dt =\iOT \left(u(t,x)-u_0(x)\right)\partial_t(\left(\zeta_\et\right)_\delta) \,dx\,dt.
\end{equation}
  Since modular convergence entails weak one and $\{( \zeta_{\et})_\delta \}_{\delta }$ is  uniformly bounded in $L^\infty $, the Lebesgue Dominated Convergence Theorem enables to pass to the limit with  $\delta\to 0$.  In turn, we obtain
\begin{equation}\label{1sttest}
\begin{split}
\iOT A\cdot \nabla \zeta_{\et} \,dx\,dt-\iOT F \zeta_{\et} \,dx\,dt&=\iOT \left(u(t,x)-u_0(x)\right)\frac{1}{{\et}}\left(\zeta(t+{\et},x)-\zeta(t,x)\right)dxdt=  \frac{1}{{\et}}\left(J_1+J_2+J_3\right),
\end{split}
\end{equation}
where $\zeta(t,x)=0$ for $t>T$, $\xi$ is extended by~\eqref{xi:ext} $u(t,x)=u_0(x)$ for $t<0$, and
\begin{eqnarray}
J_1&=&\int_0^T\iO \zeta(t+{\et},x)u(t,x)dxdt=\int_{\et}^T\iO \zeta(t,x)u(t-{\et},x)dxdt\nonumber,\\
J_2&=&-\int_0^T\iO \zeta(t,x)u(t,x)dxdt,\label{J2}\\
J_3&=&-\int_0^T\iO \zeta(t+{\et},x)u_0(x)\,dxdt+\int_0^T\iO  \zeta(t,x) u_0(x)\,dxdt=  \int_0^{\et} \iO \zeta(t,x)  u(t-{\et},x)dxdt.\label{J3}
\end{eqnarray}
Using~\eqref{J2} and~\eqref{J3} in~\eqref{1sttest} we get
\begin{equation}\label{1sttest-appl}
\begin{split}
 \iOT A\cdot \nabla \zeta_{\et} \,dx\,dt-\iOT F \zeta_{\et} \,dx\,dt&=\iOT \frac{1}{{\et}} \zeta(t,x) \left(u(t-{\et},x)-u(t,x)\right)dxdt.
\end{split}
\end{equation}
Note that for any $s_1,s_2\in\R$ we have
\begin{equation}
\label{h1-conv}
\int_{s_1}^{s_2} h_1(T_k(\s))d\s \geq  {h_1}(T_k(s_1)) (s_2-s_1).
\end{equation}
Then
\begin{equation*}%\label{est-integr1}
\begin{split}
\frac{1}{{\et}} \iOT  \zeta(t,x) \left(u(t-{\et},x)-u(t,x)\right)dxdt\leq \frac{1}{{\et}} \iOT  \xi(t,x) \int_{u(t,x)}^{u(t-{\et},x)} h_1(T_k(\s))d\s\ dx\,dt.
\end{split}
\end{equation*}
Applying it in~\eqref{1sttest-appl}, following the same reasoning as in~\eqref{J3}, we get
\begin{equation}\label{est-integr2}
\begin{split}
& \iOT A\cdot \nabla \zeta_{\et} \,dx\,dt-\iOT F \zeta_{\et} \,dx\,dt\leq  \frac{1}{{\et}} \iOT  \xi(t,x) \left(\int_{u(t,x)}^{u(t-{\et},x)} h_1(T_k(\s))d\s\right) dx\,dt=\\
&\qquad\qquad\qquad\qquad\qquad= \frac{1}{{\et}} \iOT ( \xi(t+{\et},x)-\xi(t,x))\left(\int_{u(0,x)}^{u(t-{\et},x)} h_1(T_k(\s))d\s\right) dx\,dt .
\end{split}
\end{equation}
Passing to a subsequence if necessary, we have $\zeta_{\et}\xrightharpoonup[]{*}\xi h_1(T_k(u))$ weakly-* in $L^\infty(\OT)$, when ${\et}\searrow 0$. Since $\nabla \zeta_{\et}=[(\nabla \xi) h_1(T_k(u))]_{\et} +[\xi \nabla ( h_1(T_k(u)))]_{\et} $ and $[(\nabla\xi) h_1(T_k(u))]_{\et} \xrightharpoonup[]{*}(\nabla \xi) h_1(T_k(u))$ weakly-* in $L^\infty(\OT;\rn)$, when ${\et}\searrow 0$.

Unlike the case of the modular function independent of the time variable, here the following modular convergence in $L_M(\OT;\rn)$ \[\nabla \zeta_d \xrightarrow[d\searrow 0]{M} \nabla (\xi h_1(T_k(u)))\]
is not a direct consequence of the Jensen inequality as it was in the case of the modular function independent of the time variable. To justify it we need Proposition~\ref{prop:time-app} (provided in Appendix).

 Moreover, $\zeta_d \xrightharpoonup[]{*} \xi h_1(T_k(u))$ weakly-* in $L^\infty(\OT)$, when ${\et}\searrow 0$. Therefore,  passing to the limit in~\eqref{est-integr2}  implies
\begin{equation}\label{est-integr-lim}
\begin{split}
& \iOT A\cdot \nabla (h_1(T_k(u))\xi) \,dx\,dt-\iOT F (h_1(T_k(u))\xi) \,dx\,dt\leq \iOT \partial_t\xi \int_{u_0}^{u(t,x)} {h_1}(T_k(\s))\,d\s\  dx\,dt.
\end{split}
\end{equation}

Since $T_k(u_0)\in L^\infty(\Omega),$ there exists a sequence $\{u_{0}^{n}\}_n\subset{C_c^\infty}(\Omega)$ such that $T_k(u_{0}^{n})\to T_k(u_0)$ in~$L^1(\Omega)$ and a.e. in~$\Omega$ as $n\to\infty$. For $t<0$ and all $x\in\Omega$ we put $u(t,x)=u_{0}(x)$. Recall that we consider nonnegative $\xi\in  C_c^\infty([0,T)\times\Omega )$ extended by~\eqref{xi:ext}. Note that the sequence $ \{(\wt{\zeta}_{\et})_\delta\}$ approximating $\wt{\zeta}_{\et}$, given by~\eqref{t-zeta-d}, can be used as a test function in~\eqref{eq:prop:int-by-parts-1}. Via arguments of~\eqref{1sttestve}, we pass to the limit with   $\delta\to 0$ getting
\begin{equation*}%\label{1sttest'}
\begin{split}
 \iOT A\cdot \nabla \wt{\zeta}_{\et} \,dx\,dt-\iOT F \wt{\zeta}_{\et} \,dx\,dt&=\iOT \frac{1}{{\et}}\left(\zeta(t,x)-\zeta(t-{\et},x)\right)\left(u(t,x)-u_0(x)\right)=  \frac{1}{{\et}}\left(K_1+K_2+K_3\right),
\end{split}
\end{equation*}
where
\begin{eqnarray}
K_1&=&\int_0^T\iO \zeta(t ,x)u(t,x)dxdt=\int_{\et}^{T+{\et}}\iO \zeta(t-{\et},x)u(t-{\et},x)dxdt\nonumber,\\
K_2&=&-\int_0^T\iO \zeta(t-{\et},x)u(t,x)dxdt,\label{K2}\\
K_3&=&-\int_0^T\iO \zeta(t ,x)u_0(x)\,dxdt+\int_0^T\iO  \zeta(t-{\et},x) u_0(x)\,dxdt =\int_0^{\et} \iO \zeta(t-{\et},x)  u_0(x)dxdt.\label{K3}
\end{eqnarray}
Therefore,~\eqref{K2} and~\eqref{K3} give
\begin{equation}\label{2ndtest}
\begin{split}
 \iOT A\cdot \nabla \wt{\zeta}_{\et} \,dx\,dt-\iOT F \wt{\zeta}_{\et} \,dx\,dt&= \frac{1}{{\et}}\left(L_1+L_2\right),
\end{split}
\end{equation}
with
\begin{eqnarray*}
L_1&=&\int_{\et}^T\iO \zeta(t-{\et} ,x)(u(t-{\et},x)-u(t,x))\,dxdt\nonumber,\\
L_2&=&\int_0^{\et}\iO h_1(T_k(u^n_{0}))\xi(u(t-{\et},x)-u(t,x))\,dxdt\\
&&+ \int_0^{\et}\iO (h_1(T_k(u _{0}))-h_1(T_k(u_{0}^{n})))\xi(u(t-{\et},x)-u(t,x))\,dxdt
\end{eqnarray*}
for sufficiently small ${\et}$, because $\xi(\cdot,x)$ has a compact support in $[0,T)$ almost everywhere in~$\Omega$. Due to~\eqref{h1-conv}, we have
\begin{equation}
\label{b-eta}
\begin{split}\int_{u(t-{\et},x)}^{u(t,x)}-(h_1(T_k(\s)))\,d\s&\leq -(u(t,x)-u(t-{\et},x))h_1(T_k(u(t-{\et},x)))\quad \text{ a.e. in }({\et},T)\times\Omega,\\
\int_{u(t-{\et},x)}^{u(t,x)}-(h_1(T_k(\s)))\,d\s  & \leq -(u(t,x)-u_0)h_1(T_k(u_0)) \quad  \text{ a.e. in }(0,{\et})\times\Omega.\end{split}
\end{equation}

Combining~\eqref{2ndtest} and~\eqref{b-eta} we get\begin{equation*}%\label{2ndtesti}
\begin{split}
& \iOT A\cdot \nabla \wt{\zeta}_{\et} \,dx\,dt- \iOT F\wt{\zeta}_{\et} \,dx\,dt\\
& \geq    \frac{1}{{\et}} \iOT  \xi(t,x) \left(\int_{u(t,x)}^{u(t-{\et},x)} h_1(T_k(\s))d\s \right) dx\,dt+\int_0^{\et} \iO (h_1(T_k(u_{0}))-h_1(T_k(u_{0}^{n})))\xi(u_0-u(t,x)) dx\,dt\\
& \geq  \frac{1}{{\et}} \iOT ( \xi(t-{\et},x)-\xi(t,x)) \left(\int_{u(t,x)}^{u(t-{\et},x)} h_1(T_k(\s))d\s \right) dx\,dt\\
&  -  \iO |h_1(T_k(u_{0}^{n}))-h_1(T_k(u_{0}))||\xi|(|u_0|+|u(t,x)|) dx\,dt.
\end{split}
\end{equation*}

To pass with ${\et}\searrow 0$ and then $n\to\infty$ on the left-hand side above, as in~\eqref{est-integr-lim}, by the Lebesgue Dominated Convergence Theorem we obtain\begin{equation}\label{est-integr-lim-2}
\begin{split}
& \iOT A\cdot \nabla (h_1(T_k(u))\xi) \,dx\,dt-\iOT F (h_1(T_k(u))\xi) \,dx\,dt\geq\frac{1}{{\et}} \iOT \partial_t\xi  \int_{u_0}^{u(t,x)} {h_1}(T_k(\s))\,d\s\  dx\,dt.
\end{split}
\end{equation}

Combining~\eqref{est-integr-lim} with~\eqref{est-integr-lim-2} we conclude
\begin{equation}\label{end}
\begin{split}
& \iOT A\cdot \nabla (h_1(T_k(u))\xi) \,dx\,dt-\iOT F (h_1(T_k(u))\xi) \,dx\,dt=\frac{1}{{\et}} \iOT \partial_t\xi  \int_{u_0}^{u(t,x)} {h_1}(T_k(\s))\,d\s\  dx\,dt.
\end{split}
\end{equation}
for all nondecreasing and Lipschitz $h_1:\R\to\R$ and for all nonnegative $\xi$ satisfying {\it i)} or {\it ii)}, respectively.

We can replace $h_1(T_k(u))$ by  $-h_2(T_k(u))$ in~\eqref{end} and in turn we can also replace it by  $h(T_k(u))=h(u).$ For $\xi$ satisfying {\it i)} or {\it ii)}, we have $\xi=\xi^++\xi^-$, where $\xi^+,\xi^-\in \VTMi$ or in $C_c^{\infty}([0,T)\times\Omega)$, respectively for  {\it i)}, {\it ii)}, which leads to the claim.\end{proof}

\subsection{Comparison principle}

The comparison principle we provide below is the consequence of choice proper family of test functions. The result is applied in the proof of uniqueness of solutions.

\begin{theo}\label{prop:comp-princ} 
Suppose that assumptions of Proposition~\ref{prop:intbyparts} are satisfied and for $i=1,2$ functions $v^i$ satisfy~\eqref{claim:prop:int-by-parts-1} with $  f^i\in L^1(\OT)$, $ v^i_0\in L^1(\Omega),$ and additionally
\begin{equation}\label{decay}\int_{ \{l<|v^i|<l+1\}}A(t,x,\nabla v^i)\cdot\nabla v^i\, dx\,dt\xrightarrow[l\to\infty]{} 0.\end{equation}If $f^1\leq f^2$ a.e. in~$\OT$ and $ v^1_0 \leq v^2_0$ in~$\Omega,$ then also $v^1\leq v^2$ a.e. in~$\OT$.
\end{theo}

\begin{proof} Let us define two-parameter family of functions $\bt:\R\to\R$ (with $\tau\in (0,T)$ and $r>0$ is such that $\tau+r<T$) and one-parameter family of~functions $H_\delta:\R\to\R$ (with $\delta\in(0,1)$) by
\begin{equation*}\begin{array}{ccc}
\bt(s):=\left\{\begin{array}{ll}
1& \text{for }s\in[0,\tau],\\
\frac{-s+\tau+r}{r}& \text{for }s\in[\tau,\tau+r],\\
0& \text{for }s\in[\tau+r,T]
\end{array}\right.&\quad\text{and}\quad&
H_\delta(s)=\left\{\begin{array}{ll}
0,& s\leq 0,\\
s/\delta,& s\in(0,\delta),\\
1,& s\geq \delta.
\end{array}\right.
\end{array}
\end{equation*}
 and sets
\[Q_T^\delta=\{(t,x): 0<T_{l+1}(v^1)-T_{l+1}(v^2)<\delta\},\qquad Q_T^{\delta+}=\{(t,x):  T_{l+1}(v^1)-T_{l+1}(v^2)\geq \delta\}.\]

Using \eqref{claim:prop:int-by-parts-1} with \[\left\{\begin{array}{l}h(v^1)=\psi_l(v^1) \\ \xi=H_\delta(T_{l+1}(v^1)-T_{l+1}(v^2))\bt(t)\end{array}\right.\qquad\text{ and }\qquad \left\{\begin{array}{l}h(v^2)=\psi_l(v^2) \\ \xi=H_\delta(T_{l+1}(v^1)-T_{l+1}(v^2))\bt(t)\end{array}\right.\] and subtract the second from the first we get
 \[\begin{split}
&D^{\delta,r,l,\tau}_1+D^{\delta,r,l,\tau}_2+D^{\delta,r,l,\tau}_3+D^{\delta,r,l,\tau}_4+D^{\delta,r,l,\tau}_5=\\
=&-\int_{\OT} \left(\int_{v^1_0}^{v^2_0} \psi_l(\sigma)d\sigma+\int_{v^2}^{v^1} \psi_l(\sigma)d\sigma\right) \partial_t (H_\delta( T_{l+1}(v^1)-T_{l+1}(v^2))) \bt(t)\, dx\,dt+\\
&+\iO\frac{1}{r}\int_\tau^{\tau+r}  \left(\int_{v^1_0}^{v^2_0} \psi_l(\sigma)d\sigma+\int_{v^2}^{v^1} \psi_l(\sigma)d\sigma\right) H_\delta ( T_{l+1}(v^1)-T_{l+1}(v^2))\,dt\, dx +\\
&+\int_{Q_T^\delta}\frac{1}{\delta}(A(t,x,\nabla v^1)(\psi_l(v^1)-\psi_l(v^2))  \nabla ((T_{l+1}(v^1)-T_{l+1}(v^2))\bt(t)) \,dx\,dt+\\
&+\int_{Q_T^\delta}\frac{1}{\delta}(A(t,x,\nabla v^1)-A(t,x,\nabla v^2))\psi_l(v^2)  \nabla ((T_{l+1}(v^1)-T_{l+1}(v^2))\bt(t)) \,dx\,dt+\\
&+\int_{Q_T^\delta\cup Q_T^{\delta+}} (A(t,x,\nabla v^1)\nabla v^1\psi_l'(v^1)-A(t,x,\nabla v^2)\nabla v^2\psi_l'(v^2))  H_\delta(T_{l+1}(v^1)-T_{l+1}(v^2)) \bt(t) \,dx\,dt=\\
=&\int_{Q_T^\delta\cup Q_T^{\delta+}} (f^1 \psi_l(v^1)-f^2 \psi_l(v^2)) H_\delta(T_{l+1}(v^1)-T_{l+1}(v^2)) \bt(t)\,dx\,dt=D^{\delta,r,l,\tau}_R.
\end{split}\] We observe that
\[\begin{split}|D^{\delta,r,l,\tau}_1|\leq& \int_{Q_T^\delta} \left|\partial_t\left(\int_{v^1_0}^{v^2_0} \psi_l(\sigma)d\sigma+\int_{v^2}^{v^1} \psi_l(\sigma)d\sigma\right)   \frac{1}{\delta}( T_{l+1}(v^1)-T_{l+1}(v^2))  \bt(t)\right| dx\,dt+\\
&+\int_{Q_T^\delta}\left|\left(\int_{v^1_0}^{v^2_0} \psi_l(\sigma)d\sigma+\int_{v^2}^{v^1} \psi_l(\sigma)d\sigma\right)  \frac{1}{\delta}( T_{l+1}(v^1)-T_{l+1}(v^2)) \partial_t\bt(t)\right| dx\,dt=\\
=&\int_{Q_T^\delta} \left| \psi_l(\partial_t (T_{l+1}(v^1))-\psi_l(\partial_t T_{l+1}(v^2)) \right|
 \frac{1}{\delta}\cdot\delta \, dx\,dt+\\
&+\int_{Q_T^\delta}\left| T_{l+1}(v^1_0)-T_{l+1}(v^2_0)+T_{l+1}(v^1)-T_{l+1}(v^2)\right|  \frac{1}{\delta}\cdot\delta \, dx\,dt\\
\leq& (2+4(l+1))|Q_T^\delta|.\end{split}\]
Hence,  the Dominated Convergence Theorem yields that $D^{\delta,r,l,\tau}_1\to 0$ when $\delta\to 0$.
In the case of $D^{\delta,r,l,\tau}_2$, $D^{\delta,r,l,\tau}_5$, and $D^{\delta,r,l,\tau}_R$ it also suffices to apply the Dominated Convergence Theorem. The monotonicity of $A$ implies $D^{\delta,r,l,\tau}_4\geq 0$.  We notice in the case of $D_3^{\delta,r,l,\tau}$ that on $\{|v^1|\geq  l+1\}$ we have $A(t,x,\nabla T_{l+1}(v^1))=0$. Therefore,
\[\begin{split}D_3^{\delta,r,l,\tau}& =\int_{Q_T^\delta\cap \{|v^1|< l+1,\, |v^2|\geq l+1\}}\frac{1}{\delta} A(t,x,\nabla T_{l+1}(v^1))(\psi_l(v^1)-\psi_l(v^2))  \nabla (T_{l+1}(v^1)-T_{l+1}(v^2))\bt(t) \,dx\,dt\\
&+\int_{Q_T^\delta\cap \{|v^1|< l+1,\, |v^2|< l+1\}}\frac{1}{\delta} A(t,x,\nabla T_{l+1}(v^1))(\psi_l(v^1)-\psi_l(v^2))  \nabla (T_{l+1}(v^1)-T_{l+1}(v^2))\bt(t) \,dx\,dt\\
 &=\int_{Q_T^\delta\cap \{|v^1|< l+1,\, |v^2|\geq l+1\}}\frac{1}{\delta} A(t,x,\nabla T_{l+1}(v^1)) \psi_l(v^1) \nabla  T_{l+1}(v^1) \bt(t) \,dx\,dt\\
&+\int_{Q_T^\delta\cap \{|v^1|< l+1,\, |v^2|< l+1\}}\frac{1}{\delta}(A(t,x,\nabla T_{l+1}(v^1))(\psi_l(v^1)-\psi_l(v^2))  \nabla (T_{l+1}(v^1)-T_{l+1}(v^2))\bt(t) \,dx\,dt,\end{split}\]
where the first term on the right-hand side is clearly nonnegative, while the second is tending to zero as $\delta\to 0$. Indeed, notice that $\psi_l$ is a Lipschitz function (with a Lipschitz constant $1$) and that over this set $T_{l+1}(v^1)=v^1$ and $T_{l+1}(v^2)=v^2$, and thus
 \[\begin{split}
&\left|\int_{Q_T^\delta\cap \{|v^1|< l+1,\, |v^2|< l+1\}}\frac{1}{\delta} A(t,x,\nabla T_{l+1}(v^1))(\psi_l(v^1)-\psi_l(v^2))  \nabla  (T_{l+1}(v^1)-T_{l+1}(v^2))\bt(t)  \,dx\,dt\right|\\
&\leq \int_{Q_T^\delta\cap \{|v^1|< l+1,\, |v^2|< l+1\}}\frac{1}{\delta}|A(t,x,\nabla  v^1 )|\cdot | v^1 - v^2 |\cdot|\nabla \big( v^1 - v^2 \big)|  \,dx\,dt\\
&\leq \int_{Q_T^\delta\cap \{|v^1|< l+1,\, |v^2|< l+1\}} |A(t,x,\nabla v^1)|\cdot |\nabla \big( v^1 - v^2 \big)| \,dx\,dt.
\end{split}\]
Therefore it suffices to notice that we integrate above an $L^1$-function over a shrinking domain. In turn, $\lim_{\delta\to 0}D^{\delta,r,l,\tau}_3\geq 0$.

We erase nonnegative terms on the left-hand side and pass to the limit with $\delta\to 0$ in the remaining ones, getting
 \[\begin{split}
&\lim_{\delta\to 0}D^{\delta,r,l,\tau}_2+\lim_{\delta\to 0}D^{\delta,r,l,\tau}_5=D^{ r,l,\tau}_2+ D^{ r,l,\tau}_5=\\
=&\iO\frac{1}{r}\int_\tau^{\tau+r}  \left(\int_{v^1_0}^{v^2_0} \psi_l(\sigma)d\sigma+\int_{v^2}^{v^1} \psi_l(\sigma)d\sigma\right) \sg ( T_{l+1}(v^1)-T_{l+1}(v^2))\,dt\, dx +\\
&+\iOT (A(t,x,\nabla v^1)\nabla v^1\psi_l'(v^1)-A(t,x,\nabla v^2)\nabla v^2\psi_l'(v^2))  \sg(T_{l+1}(v^1)-T_{l+1}(v^2)) \bt(t) \,dx\,dt\leq\\
\leq&\iOT (f^1 \psi_l(v^1)-f^2 \psi_l(v^2)) \sg(T_{l+1}(v^1)-T_{l+1}(v^2)) \,dx\,dt=\lim_{\delta\to 0}D^{\delta,r,l,\tau}_R.
\end{split}\]

What is more, due to~\eqref{decay} and uniform boundedness of the rest expression in $D^{r,l,\tau}_5$, we infer that $\lim_{l\to\infty}D^{r,l,\tau}_5=0.$ The Monotone Convergence Theorem enables to pass with $l\to\infty$ also in $D^{r,l,\tau}_2$ and $D^{r,l,\tau}_5$. Consequently, we obtain
 \[ \iO\frac{1}{r}\int_\tau^{\tau+r}  \left( {v^2_0}-{v^1_0} +{v^1}-{v^2} \right) \sg ( v^1-v^2)\,dt\, dx \leq \iOT (f^1  -f^2  ) \sg( v^1 -v^2) \,dx\,dt.\]
 Since a.e. $\tau\in[0,T)$ is the Lebesgue point of the integrand on the left-hand side and we can pass with $r\to 0$. After rearranging terms it results in
\[\begin{split}&\iO(v^1(\tau,x)-v^2(\tau,x))\sg(v^1(\tau,x)-v^2(\tau,x))dx\leq\\
&\leq\iOT(f^1(t,x)-f^2(t,x))\sg(v^1(\tau,x)-v^2(\tau,x))dx\,dt\\
&\quad +\iO(v_0^1(\tau,x)-v_0^2(\tau,x))\sg(v^1(\tau,x)-v^2(\tau,x))dx.\end{split} \]
for a.e. $\tau\in (0,T)$. Note that the left-hand side is nonnegative. Since $f^2\geq f^1$ and $v_0^2\geq v_0^1$, the right-hand side is nonpositive. Hence,  $\sg(v^1(\tau,x)-v^2(\tau,x))=0$ for a.e. $\tau\in (0,T)$ and~consequently $v^1 \leq v^2 $ a.e. in~$\OT$.\end{proof}

%%%%%%%%%%%%%%%%%%%%%%%%%%%%%%%%%%%%%%%%%%%%%%%%%%%%%%%%%%%%%%%%%%%%%%%%%%%%%
\subsection{Existence of a weak solution for the regularized problem}

 We apply the result of~\cite{ElMes} providing existence to a problem in the isotropic Orlicz-Sobolev setting (with the modular function depending on the norm of the gradient of solution only).  To avoid introducing overwhelming notation, we give here only direct simplification of \cite[Theorem~2]{ElMes} to our situation. It reads as follows.
 
 \begin{coro}\label{coro:ElMes} Let $\Omega$ be a bounded open domain in $\rn$ and $m:\rn\to\R$ be a radially summetric (isotropic) function, i.e. $m(\xi)=\overline{m}(|\xi|)$ with some  $\overline{m}:\R\to\R$. Denote an Orlicz space $W^1 L_{\overline{m}}(\OT)=\{u\in L_{\overline{m}}(\OT):\ u\in L_{\overline{m}}(\OT)\}.$ We consider a Carath\'{e}odory function $a:\Omega\times \rn\to\rn$, which is strictly monotone, i.e.
\begin{equation}\label{(9)ElMes}
\forall_{\xi_1,\xi_2\in\rn}\qquad\Big(a(x,\xi_1)-a(x,\xi_2)\Big)(\xi_1-\xi_2)>0
\end{equation}
and satisfies growth and coercivity conditions
\begin{equation}\label{(10)i(7)ElMes}
\exists_{c_0,c_1,c_2>0}\ \forall_{(x,\xi)\in \Omega\times \rn}\qquad a(x,\xi)\xi\geq c_0 m(\xi)\qquad\text{and}\qquad |a(x,\xi)|\leq c_1 (m^*)^{-1}\big(m(c_2\xi)\big).
\end{equation}  

Then the problem \begin{equation}\label{eq:ElMes}
\left\{\begin{array}{ll}
\partial_t v-\dv\, a(x,\nabla v)= g\in L^\infty(\OT) & \ \mathrm{ in}\  \OT,\\
v(t,x)=0 &\ \mathrm{  on} \ (0,T)\times\partial\Omega,\\
v(0,\cdot)=v_{0}(\cdot)\in L^2(\Omega) & \ \mathrm{ in}\  \Omega.
\end{array}\right.
\end{equation}
with  has at least one weak solution $v\in W^1 L_{\overline{m}}(\OT)\cap C([0,T],L^2(\Omega)).$

Moreover, the energy equality is satisfied, i.e.
\begin{equation}
\label{en:eq:ElMes}
\frac{1}{2}\iO (v(\tau))^2\, dx-
\frac{1}{2}\iO (v_{0} )^2\, dx+\int_{\Omega_\tau} a(x,\nabla v)\cdot \nabla v\, dx\, dt=  \int_{\Omega_\tau} g v\, dx\, dt\qquad\forall_{\tau\in[0,T].}\end{equation}
 \end{coro}

The application of the above result gives the following proposition yields the existence of solutions to a regularized problem.

The following proposition yields the existence of solutions to a regularized problem.
\begin{prop}\label{prop:reg-bound}
Let a locally integrable $N$-function $M$ satisfy assumption ($\mathcal{M}$) (resp.~($\mathcal{M}_p$)) and function $A$ satisfy assumptions ($\mathcal{A}$\ref{A1})--($\mathcal{A}$\ref{A3}). Assume that $m$ is an~isotropic function growing essentially more rapidly than $M$ and $m(|\xi|)\geq M(\xi)$ for all $\xi\in\rn$. We consider a regularized operator given by
\begin{equation}\label{Atheta}
A_\theta(t,x,\xi):=A(t,x,\xi)+\theta \bn m(\xi)\qquad \forall_{x\in\rn,\,\xi\in\rn}.
\end{equation}

Let $f\in L^\infty(\Omega_T)$ and $u_0\in L^2(\Omega)$. Then for every $\theta\in(0,1]$ there exists a weak solution to the problem \begin{equation}\label{eq:reg-bound}
\left\{\begin{array}{ll}
\partial_t u^\theta-\dv A_\theta(t,x,\nabla u^\theta)= f & \ \mathrm{ in}\  \OT,\\
u^\theta(t,x)=0 &\ \mathrm{  on} \ (0,T)\times\partial\Omega,\\
u^\theta(0,\cdot)=u_{0}(\cdot) & \ \mathrm{ in}\  \Omega.
\end{array}\right.
\end{equation} Namely, there exists $u^\theta \in C([0,T];L^2(\Omega))\cap L^1(0,T; W^{1,1}_{0}(\Omega))$, such that
\begin{equation}\begin{split}\label{weak-reg-bound}
&-\iOT u^\theta\partial_t \vp dx\, dt+\iO u^\theta(T)\vp(T)\, dx- \iO u^\theta(0)\vp(0)\, dx+\iOT A_\theta(t,x,\nabla u^\theta)\cdot \nabla \vp\, dx\, dt  = \iOT f\vp\, dx\, dt
\end{split}\end{equation}
holds for $\vp\in C^\infty([0,T];C_c^\infty(\Omega))$.

\medskip

 Furthermore, \begin{itemize}\item the family $\big(u^\theta\big)_\theta$ is uniformly bounded in $L^\infty(0,T;L^2(\Omega))$,\\
\item the family $\big(\nabla u^\theta\big)_\theta$ is uniformly integrable in $ L_M(\Omega_T;\rn)$,\\
\item the family $\big(A(t,x,\nabla u^\theta)\big)_\theta$ is uniformly bounded in $ L_{M^*}(\Omega_T;\rn)$,\\
\item the family and $\big(\theta m^*(\bar\nabla m(|{\nabla} u ^\theta|))_\theta$ is uniformly bounded in $ L^1(\Omega_T)$.\end{itemize} 

\medskip

Moreover, the following energy equality is satisfied
\begin{equation}
\label{en:eq}
\frac{1}{2}\iO (u^\theta(\tau))^2\, dx-
\frac{1}{2}\iO (u_{0} )^2\, dx+\int_{\Omega_\tau} A_\theta(t,x,\nabla u^\theta)\cdot \nabla u ^\theta\, dx\, dt=  \int_{\Omega_\tau} f u^\theta\, dx\, dt, \qquad \tau\in[0,T].
\end{equation}
\end{prop}
\begin{proof} To get existence we apply Corollary~\ref{coro:ElMes}, whereas a priori estimates results from the analysis of the structure of regularization. 

\medskip

\textbf{Existence.}  Recall that we use notation $\nabla$ for a gradient with respect to the spacial variable. Let us introduce also notation $\bn:=\nabla_\xi$. Using it  $\bn m(\xi) =\nabla_\xi \overline{m}(|\xi|)=\xi\overline{m}'(|\xi|)/|\xi|$. Observe that it gives equality in the Fenchel--Young inequality in the following way
\begin{equation}\label{FYeq}
\bn m(\xi)\cdot\xi = \overline{m}( |\xi|)+\overline{m}^*(|\bn m(\xi)|).
\end{equation} Since we take an $N$-function $m$ which grows essentially more rapidly than $M$ we observe that  $m$ is strictly monotone as a gradient of a strictly convex function, i.e.
\begin{equation}\label{msmon}
(\bn m(\xi)-\bn m(\eta))(\xi-\eta)>0\qquad \forall_{\xi,\eta\in\rn}.
\end{equation}

 To apply Corollary~\ref{coro:ElMes} we shall show~\eqref{(10)i(7)ElMes}. The coercivity condition results directly from ({\cal A}2). The bound on growth follows from the Fenchel-Young inequality~\eqref{inq:F-Y},~\eqref{FYeq}, and ({\cal A}3)
 \[  c_A \overline{m}^*\left(\frac{1}{2} \left|A_\theta(x,\xi)\right| \right)\leq \overline{m}\left(\left|\frac{2}{c_A}\xi\right|\right) \]
and further, by convexity of $\overline{m}^*$,
\[ \left|A_\theta(x,\xi)\right|  \leq 2(\overline{m}^*)^{-1}\left(\frac{1}{c_A} \overline{m}\left(\left|\frac{2}{c_A}\xi\right|\right)\right)\leq \frac{2}{c_A}(\overline{m}^*)^{-1}\left( \overline{m}\left(\left|\frac{2}{c_A}\xi\right|\right)\right). \]
Therefore, Corollary~\ref{coro:ElMes} (coming from \cite[Theorem~2]{ElMes})

It suffices to show 
\begin{equation*}\overline{m}^*(|c_1 A_\theta(t,x,\xi)|)\leq \overline{m}(|c_2 \xi|),
%\label{A-ElMes7suf}
\end{equation*}
which follows from equality~\eqref{FYeq} in the Fenchel-Young inequality~\eqref{inq:F-Y} and ($\mathcal{A}3$) where $c_A,\theta\in(0,1]$. We have
\[A_\theta(t,x,\xi)\cdot \xi\leq \overline{m}\left(\left|\frac{2}{c_A}\xi\right|\right)+\overline{m}^*\left(\left|\frac{c_A}{2}A_\theta(t,x,\xi)\right|\right) \leq \overline{m}\left(\left|\frac{2}{c_A}\xi\right|\right)+c_A \overline{m}^*\left(\left|\frac{1}{2}A_\theta(t,x,\xi)\right|\right),\] but on the other hand
 \begin{equation*}
 \begin{split}
 A_\theta(t,x,\xi)\cdot \xi &\geq  M\left(t,x,\xi\right)+\theta \overline{m}\left(\left|\xi\right|\right)+\theta \overline{m}^*\left(\left|\bn m(\xi)\right|\right)\\
 &\geq c_A M^*\left(t,x,A(t,x,\xi)\right)+0 +\theta \overline{m}^*\left(\left|\bn m(\xi)\right|\right)\\
 &\geq 2 c_A\left(\frac{1}{2}\overline{m}^*\left(\left|A(t,x,\xi)\right|\right) +\frac{1}{2}\overline{m}^*\left(\left|\theta\bn m(\xi)\right|\right) \right) \geq 2 c_A \overline{m}^*\left(\frac{1}{2} \left|A_\theta(t,x,\xi)\right| \right).
 \end{split}
 \end{equation*}
 Therefore, we get
$  c_A \overline{m}^*\left(\frac{1}{2} \left|A_\theta(t,x,\xi)\right| \right)\leq \overline{m}\left(\left|\frac{2}{c_A}\xi\right|\right) $. Then by strict monotonicity of $\overline{m}^*$,
\[ \left|A_\theta(t,x,\xi)\right|  \leq \frac{2}{c_A}(\overline{m}^*)^{-1}\left( \overline{m}\left(\left|\frac{2}{c_A}\xi\right|\right)\right)  \]
and Corollary~\ref{coro:ElMes}, gives the claim, i.e. existence of a solution $u^\theta \in C([0,T];L^2(\Omega))\cap L^1(0,T; W^{1,1}_{0}(\Omega))$ with $\nabla u^\theta\in L_m(\Omega_T;\rn)$.

\bigskip

Now, we shall show uniform boundedness of $u^\theta$.

\medskip

\textbf{A priori estimates.} By energy equality~\eqref{en:eq}, (A2) and~\eqref{FYeq} we get
\begin{equation}
\label{en:eq:imp}
\begin{split} \frac{1}{2}\iO (u^\theta(\tau))^2\, dx+
 \int_{\Omega_\tau} M(t,x,\nabla u^\theta)\, dx\, dt +\int_{\Omega_\tau}\theta m( \nabla u^\theta )+\theta m^*(\bar{\nabla} m(  {\nabla} u^\theta ))\,dx\,dt \\
 \leq  \int_{\Omega_\tau} f u^\theta\, dx\, dt+\frac{1}{2} \iO (u_{0})^2dx.\end{split}
\end{equation}
We estimate further the right-hand side using the Fenchel-Young inequality~\eqref{inq:F-Y} and the~modular Poincar\'{e} inequality (Theorem~\ref{theo:Poincare}). For this let us consider any $N$-function $B:[0,\infty)\to[0,\infty)$ such that \[B(s)\leq \frac{1}{2c^2_P}\inf_{(t,x)\in\OT,\,\xi:\,|\xi|=s}M(t,x,\xi),\] where $c^2_P$ is the constant from the modular Poincar\'{e} inequality for $B$. Then on the right-hand side of~\eqref{en:eq:imp} we have
\[\begin{split}
\int_{\Omega_\tau} f u^\theta\, dx\, dt&\leq \int_{\Omega_\tau} B^*(\|f\|_{L^\infty(\Omega)}/c_1^P)\, dx\, dt+\int_{\Omega_\tau} B(c_1^P|u^\theta|)\, dx\, dt \\
&\leq |\Omega_T| B^*(\|f\|_{L^\infty(\Omega)}/c_1^P)+c_P^2\int_{\Omega_\tau} B(|\nabla u^\theta|)\, dx\, dt\\
&\leq c(\Omega,T,f,N)+\frac{1}{2}\int_{\Omega_\tau} M(t,x,\nabla u^\theta )\, dx\, dt.
\end{split}\]
Consequently, we infer that~\eqref{en:eq:imp} implies  \begin{equation*}
%\label{apriori1}
\begin{split} \frac{1}{2}\iO (u^\theta(\tau))^2\, dx+
\int_{\Omega_\tau} \frac{1}{2} M(t,x,\nabla u^\theta)\, dx\, dt+ \int_{\Omega_\tau}\theta m( \nabla u^\theta )+\theta m^*(\bar\nabla m( {\nabla} u^\theta ))\,dx\,dt\\ \qquad\qquad\leq  c(\Omega,T,f,N)+\frac{1}{2} \|u_{0}\|_{L^2(\Omega)} \leq c(\Omega,T,f,N,\|u_{0}\|_{L^2(\Omega)} )=\wt{C}.\end{split}
\end{equation*}
When we take into account that $\tau$ is arbitrary, this observation leads to a priori estimates
\begin{eqnarray}
&\sup_{\tau \in[0,T]}\|u ^\theta(\tau)\|^2_{L^2(\Omega)}&\leq \wt{C},\label{apriori:utL2} \\
& \int_{\OT}  M(t,x,\nabla u ^\theta)\,dx\,dt&\leq 2\wt{C},\label{apriori:Mt}\\
%& \int_{\OT} \theta m(\nabla u ^\theta)\,dx\,dt&\leq 2\wt{C},\nonumber\\
&\int_{\OT}  \theta m^*(\bar\nabla m(|{\nabla} u ^\theta|))\,dx\,dt&\leq \wt{C}.\label{apriori:mst}
\end{eqnarray}
Moreover, (${\cal A}$\ref{A2}) implies then
\begin{equation}
c_A\int_{\OT} M^*(t,x, A (t,x,\nabla u ^\theta))\,dx\,dt \leq 2\wt{C}.\label{apriori:Mst}
\end{equation}
And thus, the uniform boundednesses of the claim follow.
\end{proof}

\section{The main proof}

We prove the existence of a weak solution for non-regularized problem with bounded data by passing to the limit with $\theta\to 0$ in the regularized problem~\eqref{eq:reg-bound}. Note that to get weak solutions, we exploit the integration-by-parts-formula from Proposition~\ref{prop:intbyparts}. In particular, we require regularity of $M$ necessary in approximation used in Proposition~\ref{prop:intbyparts}.

\begin{proof}[Proof of Theorem~\ref{theo:bound}] We apply Proposition~\ref{prop:reg-bound} and  let $\theta\to 0$. Uniform estimates provided therein imply that there exist a subsequence of $\theta\to 0$, such that
\begin{eqnarray}
\label{conv:untLi}  u^\theta\xrightharpoonup* u\quad& \text{weakly-* in }& L^\infty(0,T;L^2(\Omega)),\\
\label{limDut}\nabla u^\theta \xrightharpoonup* \nabla u\quad& \text{weakly-* in }& L_M(\OT;\rn),
\end{eqnarray}
with some $u\in\VTMi$ and there exists $\alpha\in L_{M^*}(\OT;\rn)$, such that
\begin{equation}
\label{limaDut}A(\cdot,\cdot,\nabla u^\theta) \xrightharpoonup* \alpha\quad \text{weakly-* in } L_{M^*}(\OT;\rn).\end{equation}

\medskip

\textbf{Identification of the limit $\alpha$. Uniform estimates.} We need to show
\begin{equation}\label{limsup<aln}\limsup_{\theta\to 0}\iOT A(t,x,\nabla u^\theta)\cdot\nabla u ^\theta\,dx\,dt\leq \iOT \alpha \nabla u \,dx\,dt.\end{equation}
We are going to pass with $\theta\searrow 0$ in the regularized problem~\eqref{weak-reg-bound} and~\eqref{en:eq}. In the first term on the left-hand side therein, due to~\eqref{conv:untLi}, we have
\begin{equation*}
%\label{lim1}
\lim_{\theta\searrow 0} \iOT u^\theta\partial_t \vp\, dx\, dt=\iOT u \partial_t \vp\, dx\, dt.
\end{equation*}
Moreover, we prove that
\begin{equation*}
%\label{thetamvanish}
\lim_{\theta\searrow 0}\iOT\theta\bar{\nabla}m(\nabla u^\theta)\cdot\nabla \vp\,dx\,dt=0.
\end{equation*}
   To get this, we split $\OT$ into $\Omega^\theta_{T,R}=\{(t,x)\in \OT:|\nabla u^\theta|\leq R\}$ and its complement and consider the following integrals separately
\[\iOT\theta\bar{\nabla}m(\nabla u^\theta)\cdot\nabla \vp\,dx\,dt=\int_{\Omega^\theta_{T,R}}\theta\bar{\nabla}m(\nabla u^\theta)\cdot\nabla \vp\,dx\,dt+\int_{\OT\setminus\Omega^\theta_{T,R}}\theta\bar{\nabla}m(\nabla u^\theta)\cdot\nabla \vp\,dx\,dt.\]
To deal with the first term on the right-hand side above, we use continuity of $\bar{\nabla} m$ to obtain
\[\lim_{\theta\searrow 0}\int_{\Omega^\theta_{T,R}}\theta\bar{\nabla}m(\nabla u^\theta)\cdot\nabla \vp\,dx\,dt\leq \lim_{\theta\searrow 0} \left(\theta|\OT|\cdot\|\nabla \vp\|_{L^\infty(\OT;\rn)}\sup_{\xi:\,|\xi|\leq R}|\bar{\nabla}m(\xi)|\right)=0.\]

As for the integral over $\OT\setminus\Omega^\theta_{T,R}$,  a priori estimate~\eqref{apriori:Mt} implies that the sequence $\{\nabla u ^\theta\}_\theta$ is uniformly bounded in $L^1(\OT)$ and thus
\begin{equation}
\label{qtrc:meas}\sup_{0\leq\theta\leq 1}|\OT\setminus\Omega^\theta_{T,R}|\leq  {C}/{R}.
\end{equation}
Furthermore, since $m^*$ is an $N$-function, for $\theta\in(0,1)$ we have $m^*(\theta\cdot)\leq\theta m^*(\cdot).$ This together with  $L^1(\OT)$-bound~\eqref{apriori:mst} for
$ \theta m^*(\bar\nabla m( {\nabla} u^\theta ))$, which is  uniform with respect to $\theta$, we get $L^1(\OT)$-bound  for
$ \{ m^*(\theta\bar\nabla m( {\nabla} u^\theta ))\}_\theta$. Therefore, Lemma~\ref{lem:unif} implies the uniform integrability of $\{\theta  \bar\nabla m( {\nabla} u^\theta)\}_\theta$. Then, using~\eqref{qtrc:meas}, we obtain
\[
\int_{\OT\setminus\Omega^\theta_{T,R}}\theta\bar{\nabla}m(\nabla u^\theta)\cdot\nabla \vp\,dx\,dt\leq \|\nabla \vp\|_{L^\infty(\OT;\rn)} \sup_{\theta\in (0,1)}\int_{\OT\setminus\Omega^\theta_{T,R}}\theta|\bar{\nabla}m(\nabla u^\theta)|\,dx\,dt\xrightarrow[R\to\infty]{} 0.
\]

Therefore,  we can pass to the limit in the weak formulation of the regularized problem~\eqref{weak-reg-bound}. Because of~\eqref{limaDut} %(note that $\nabla\vp\in \DOT\subset E_M(\OT;\rn)$)
 we  obtain
\begin{equation}
\label{lim:theta=0}-\iOT u \partial_t \vp dx\, dt- \iO u_{0}\vp(0)\, dx+\iOT \alpha\cdot \nabla  \vp\, dx= \iOT f\vp\, dx\, dt \quad\forall_{\vp\in C_c^\infty([0;T)\times \Omega)}.
\end{equation}
We want to use here a test function involving the solution itself. In order to do it, we need to apply the integration-by-parts formula from Proposition~\ref{prop:intbyparts} applied to~\eqref{lim:theta=0}  with $A=\alpha$,  $F=f$, and $h(\cdot)= T_k(\cdot)$. We obtain
\[
 -\int_{\OT} \left(\int_{u_{0} }^{u(t,x)} T_k(\s)d\s\right) \partial_t  \xi \,dx\,dt =- \int_{\OT} \alpha\cdot \nabla (T_k(u)\xi) \,dx\,dt+\int_{\OT}f T_k(u)\xi \,dx\,dt,
\]
for every $\xi\in C_c^\infty([0,T)\times \overline{\Omega})$.  Let two-parameter family of functions $\vt:\R\to\R$ be defined by
\begin{equation}
\label{vt}\vt(t):=\left(\omega_r * \mathds{1}_{[0,\tau)}\right)(t),
\end{equation}where $\omega_r$ is a standard regularizing kernel, that is $\omega_r\in C_c^\infty(\R)$, $\supp\,\omega_r\subset (-r,r)$. Note that $\supp\,\vt=[-r,\tau+r).$ In particular, for every $\tau$ there exists $r_\tau$, such that for all $r<r_\tau$ we have $\vt\in C_c^\infty([0,T))$.
Then taking~$\xi(t,x)=\vt(t)$,   we get
\begin{equation}\begin{split}\label{przedlim}
 -\int_{\OT} \left(\int_{u_{0} }^{u (t,x)} T_k(\s)d\s \right) \partial_t  \vt \,dx\,dt =-\int_{\OT}\alpha \cdot \nabla (T_k(u))\vt \,dx\,dt+\int_{\OT}fT_k(u)\vt \,dx\,dt.\end{split}
\end{equation}
On the right-hand side we integrate by parts obtaining\[ -\int_{\OT} \left(\int_{u_{0}}^{u (t,x)} T_k(\s)d\s \right) \partial_t  \vt    \,dx\,dt=\int_{\OT}\partial_t \left(\int_{u_{0}}^{u (t,x)} T_k(\s)d\s \right)   \vt \,dx\,dt.\]
Then we pass to the limit with ${r}\to 0$, apply the Fubini theorem, and integrate over the time variable\[\begin{split}\lim_{{r}\to 0} \int_{\OT}\partial_t \left(\int_{u_{0}}^{u (t,x)} T_k(\s)d\s \right)   \vt \,dx\,dt&=  \int_{\Omega_\tau}\partial_t \left(\int_{u_{0}}^{u (t,x)} T_k(\s)d\s \right)   \,dx\,dt =  \int_{\Omega } \left(\int_{u_{0}}^{u (\tau,x)} T_k(\s)d\s \right)   \,dx.\end{split}\]

Passing with ${r}\to 0$ in~\eqref{przedlim} for a.e.~$\tau\in[0,T)$ we get
\[
 \int_{\Omega} \left(\int_{0}^{u (\tau,x)} T_k(\s)d\s-\int_{0}^{u_{0}(x)} T_k(\s)d\s\right) \,dx  =-\int_{\Omega_\tau}\alpha \cdot \nabla  T_k(u) \,dx\,dt+\int_{\Omega_\tau}f T_k(u)  \,dx\,dt.
\]
Applying the Lebesgue Monotone Convergence Theorem for $k\to\infty$ we obtain
\[\frac{1}{2}\|u(\tau)\|^2_{L^2(\Omega)}-\frac{1}{2}\|u_{0}\|^2_{L^2(\Omega)}=-\int_{\Omega_\tau}\alpha  \cdot \nabla u \,dx\,dt+\int_{\Omega_\tau}f u \,dx\,dt.\]
On the other hand, considering the energy equality~\eqref{en:eq} in the first term on the left-hand side we take into account the weak lower semi-continuity of $L^2$-norm and~\eqref{conv:untLi} and realize that
\[\|u^\theta(\tau)\|^2_{L^2(\Omega)}=\lim_{\epsilon\to 0} \frac{1}{\epsilon}\int_{\tau-\epsilon}^{\tau}\|u^\theta(s)\|^2_{L^2(\Omega)}ds\geq \lim_{\epsilon\to 0}\frac{1}{\epsilon} \int_{\tau-\epsilon}^{\tau}\|u (s)\|^2_{L^2(\Omega)}ds=\|u(\tau)\|^2_{L^2(\Omega)}.\]
Erasing the nonnegative term~$\iOT\theta\bar{\nabla}m(\nabla u^\theta)\cdot\nabla u^\theta\,dx\,dt$  in~\eqref{en:eq} and then passing to the limit with $\theta\searrow 0$, we get
\begin{equation}\label{en-eq-impl} \frac{1}{2}\|u (\tau)\|^2_{L^2(\Omega)}- \frac{1}{2}\|u_{0}\|^2_{L^2(\Omega)}+\limsup_{\theta\searrow 0}\iOT A (t,x,\nabla u^\theta)\cdot \nabla u^\theta\, dx\, dt\leq \iOT f u \, dx\, dt.\end{equation}
Thus,~\eqref{limsup<aln} follows.

\medskip

\textbf{Identification of the limit $\alpha$. Conclusion by monotonicity trick.} Let us concentrate on proving  \begin{equation}
\label{lim=ca}A(t,x,\nabla u)=\alpha\qquad \text{a.e.}\quad\text{in}\quad \Omega_T.
\end{equation}
 Monotonicity assumption ($\mathcal{A}3$) of~$A$ implies
\[(A(t,x,\nabla u^\theta)-A(t,x,\eta))\cdot(\nabla u^\theta-\eta)\geq 0\qquad\text{a.e.  in }\OT, \text{ for any }\eta\in L^\infty(\OT;\rn)\subset E_M(\OT;\rn).\]
 Since $A(\cdot,\cdot,\eta)\in L_{M^*}(\OT,\rn)$, we pass to the limit with $\theta\searrow 0$ and take into account~\eqref{limsup<aln} to conclude that
\begin{equation}
\label{mono:int}
\iOT (\alpha-A(t,x,\eta))\cdot(\nabla u-\eta)\,dx\,dt\geq 0.
\end{equation}
Then Lemma~\ref{lem:mon} with ${\cal A}=\alpha$ and $\xi=\nabla u$ implies~\eqref{lim=ca}.

\medskip

\textbf{Conclusion of the proof of Theorem~\ref{theo:bound}.} We pass to the limit in  the weak formulation of bounded regularized problem~\eqref{weak-reg-bound} due to~\eqref{conv:untLi}, \eqref{limDut}, \eqref{limaDut}, and~\eqref{lim=ca}, getting the existence of~$u\in\VTMi$ satisfying
\[
-\iOT u \partial_t \vp dx\, dt- \iO u(0)\vp(0)\, dx+\iOT A (t,x,\nabla u)\cdot \nabla \vp \, dx\, dt= \iOT f\vp\, dx\, dt\quad\forall_{\vp\in C_c^\infty([0;T)\times \Omega)},\]
i.e.~\eqref{weak-bound}, which ends the proof of existence in the anisotropic case. It suffices to notice that uniqueness results from comparison priciple (Proposition~\ref{prop:comp-princ}).

\medskip

 Let us point out that in the case of any other reflexive space whenever in the proof above we apply the approximation by a~sequence of smooth functions converging in the modular, we can use instead a strongly converging  affine combination of a weakly converging sequence (ensured in   reflexive Banach spaces by Mazur's Lemma).  \end{proof}

\textbf{Proof of Theorem~\ref{theo:main0}.} It is provided in Lemma~\ref{lem:Mass} (in Appendix) that in isotropic spaces it suffices to have $({\cal M}^{iso})$ (resp. $({\cal M}^{iso}_p)$) to get $({\cal M})$ (resp. $({\cal M}_p)$). Therefore, Theorem~\ref{theo:main0} is a direct consequence of~Theorem~\ref{theo:bound}. \qed

%%%%%%%%%%%%%%%%%%%%%%%%%%%%%%%%%%%%%%%%%%%%%%%%%%%%%%%%%%%%%%%%%%%%%%%%%%%%%%%%%
\section{Appendix}\label{sec:appendix}
%%%%%%%%%%%%%%%%%%%%%%%%%%%%%%%%%%%%%%%%%%%%%%%%%%%%%%%%%%%%%%

\subsection{Classics}

\begin{lem} \label{lem:TM1}
 Suppose $w_n\xrightharpoonup[n\to\infty]{}w$ in $L^1(\OT)$, $v_n,v\in L^\infty(\OT)$, and $v_n\xrightarrow[n\to\infty]{a.e.}v$. Then \[\int_\OT w_n v_n\,dx\,dt \xrightarrow[n\to\infty]{}\int_\OT w v\,dx\,dt.\]
\end{lem}

\begin{defi}[Uniform integrability] We call a sequence  $\{f_n\}_{n=1}^\infty$ of measurable functions $f_n:\OT\to \rn$
 uniformly integrable if
\[\lim_{R\to\infty}\left(\sup_{n\in\mathbb{N}}\int_{\{x:|f_n(x,t)|\geq R\}}|f_n(x)|\,dx\,dt\right)=0.\]
 \end{defi}

%We have two equivalent definitions of modular convergence.
\begin{defi}[Modular convergence]\label{def:convmod}
We say that a sequence $\{\xi_i\}_{i=1}^\infty$ converges modularly to $\xi$ in~$L_M(\OT;\rn)$ (and denote it by $\xi_i\xrightarrow[i\to\infty]{M}\xi$), if
\begin{itemize}
\item[i)] there exists $\lambda>0$ such that
\begin{equation*}
%\label{convmodi}
\int_{\OT}M\left(t,x,\frac{\xi_i-\xi}{\lambda}\right)\,dx\,dt\to 0,
\end{equation*}
equivalently
\item[ii)] there exists $\lambda>0$ such that
\begin{equation*}
%\label{convmodii}
 \left\{M\left(t,x,\frac{\xi_i}{\lambda}\right)\right\}_i \ \text{is uniformly integrable in } L^1(\OT)\quad \text{and}\quad \xi_i\xrightarrow[]{i\to\infty}\xi \ \text{in measure};
\end{equation*}
\end{itemize}
\end{defi}

%We use the following results.
\begin{lem}[Modular-uniform integrability,~\cite{gwiazda2}]\label{lem:unif}
Let $M$ be an $N$-function and $\{f_n\}_{n=1}^\infty$ be a~sequence of measurable functions such that $f_n:\OT\to \rn$ and $\sup_{n\in\N}\int_\OT M(t,x,f_n(x))\,dx\,dt<\infty$. Then the sequence $\{f_n\}_{n=1}^\infty$ is uniformly integrable.
\end{lem}

The following result can be obtained by the method of the proof of~\cite[Theorem~7.6]{Musielak}.
\begin{lem}[Density of simple functions, \cite{Musielak}]\label{lem:dens}
Suppose~\eqref{ass:M:int}. Then the set of simple functions integrable on $\OT$ is dense in $L_M(\OT)$ with respect to the modular topology.
\end{lem}

\begin{theo}[The Vitali Convergence Theorem]\label{theo:VitConv} Let $(X,\mu)$ be a positive measure space, $\mu(X)<\infty $, and $1\leq p<\infty$. If $\{f_{n}\}$ is uniformly integrable in $L^p_\mu$,   $f_{n}(x)\to f(x)$ in measure  and $|f(x)|<\infty $  a.e. in $X$, then  $f\in  {L}^p_\mu(X)$
and  $f_{n}(x)\to f(x)$ in  ${L}^p_\mu(X)$.
\end{theo}

In the isotropic case we have the following corollary of the classical result by Gossez.
\begin{theo}[\cite{Gossez}, Theorem 4]\label{theo:approx-Gossez} Suppose $\Omega\subset\rn$, $N\geq 1$ is a bounded Lipschitz domain, $B:\rp\to\rp$, and $g\in W_0^{1,1}(\Omega)$ is such that $\int_\OT B(|\nabla \vp|)dxdt<\infty$. Then there exists   $\vp_\delta\in C_c^\infty(\Omega)$ such that
\[ g_\delta\xrightarrow[\delta\to \infty]{}g\ \text{strongly in } \ W^{1,1}(\Omega)\quad\text{and}\quad \exists_{\lambda>0}:\ \ \nabla g_\delta\xrightarrow[\delta\to 0]{mod} \nabla g\ \text{ in } L_B(\Omega).\]
\end{theo}

\subsection{Isotropic conditions}

\begin{lem}[Isotropic conditions]\label{lem:Mass} Suppose  a cube $\Qd$ and interval $I_i^\delta$ are an arbitrary ones defined in (${\cal M}$) with $\delta<\delta_0(N,M)$ and a locally integrable $N$--function $M:[0,T)\times\R\times[0,\infty)\rightarrow[0,\infty)$  satisfies condition $({\cal M}^{iso})$ or $({\cal M}^{iso}_p)$.  Let us consider function   $ \Mijd $   given by~\eqref{Mijd}, and its greatest convex minorants $(\Mijd)^{**}$. Then there exists a function  $\Theta:[0,1]\times\rp\to \rp$ (resp.  $\Theta_p:[0,1]\times\rp\to \rp$)  such that $(\mathcal{M})$ (resp. $(\mathcal{M}_p)$) is satisfied.
\end{lem}
\begin{proof} We prove the claim under assumption  $(\mathcal{M})$ and indicating the necessary slight modification in the case of  $(\mathcal{M}_p)$. 

Recall the notation introduced in Lemma~\ref{lem:step2prev}. We fix an arbitrary $y\in Q^\delta_j$ and note that
		\begin{equation}\label{Quotient}
			\frac{M(t,y,s)}{(\Mijd)^{**}(s)}=\frac{M(t,y,s)}{\Mijd(s)}\frac{\Mijd(s)}{(\Mijd)^{**}(s)}.
		\end{equation}
		We estimate separately both quotients on the right hand side of the latter equality. By continuity of $M$ we find $(\bar{t},\bar{y})\in I_i^\delta\times\widetilde{Q}^\delta_j$ such that $\Mijd(s)=M(\bar{t},\bar{y},s)$. Then for sufficiently small $\delta$ we have
		\begin{equation}\label{FirQuoEst}
		\frac{M(t,y,s)}{\Mijd(s)}=	\frac{M(t,y,s)}{M(\bar{t},\bar y,s)}\leq \Theta^{iso}(|t-\bar{t}|+c_x |x-\bar{x} | ,s)\leq \Theta^{iso}(c\delta,s) 
		\end{equation}
with a constant $c_x$ depending on the dimension only. If  $(\mathcal{M}_p)$ is assumed instead of  $(\mathcal{M})$, we have here $\Theta^{iso}_p$.		
		
		In order to estimate the second quotient in \eqref{Quotient} we observe first that if $s\in[0,\infty)$ is such that $\Mijd(s)=(\Mijd)^{**}(s)$ then the statement is obvious. Therefore, we assume that $\Mijd(s_0)>(\Mijd)^{**}(s_0)$ at some $s_0$. Due to continuity of $\Mijd$ and $(\Mijd)^{**}$ there is a neighborhood $U$ of $s_0$ such that $\Mijd>(\Mijd)^{**}$ on $U$. Consequently, $(\Mijd)^{**}$ is affine on $U$. Moreover, Definition~\ref{def:Nf} implies that $m_1\leq \Mijd\leq m_2$, where $m_1$ and $m_2$ are convex. Therefore there are $s_1,s_2$ such that $U\subset(s_1,s_2)$, $\Mijd>(\Mijd)^{**}$ on $(s_1,s_2)$, $(\Mijd)^{**}(s_i)=M^\delta_j(s_i)$, $i=1,2$ and $(\Mijd)^{**}$ is an affine function on $[s_1,s_2]$, i.e.
\begin{equation}\label{ConvexificationIsAffine}
	(\Mijd)^{**}(qs_1+(1-q)s_2)=q\Mijd(s_1)+(1-q)\Mijd(s_2),\qquad\text{for}\quad q\in[0,1].
\end{equation}
We note that we consider $s_1>0$, because it follows that $0=\Mijd(0)=(\Mijd)^{**}(0)$. Now, thanks to the continuity of $M$ we find $(t_i,y_i)\in I_i^\delta\times\wt{Q}^\delta_j$ such that $M^\delta_j(s_i)=M(t_i,y_i,s_i)$, $i=1,2$. Consequently, it follows from \eqref{ConvexificationIsAffine} that
\begin{equation}\label{ConvexificationIsAffineII}
	(\Mijd)^{**}(qs_1+(1-q)s_2)=qM(t_1,y_1,s_1)+(1-q)M(t_2,y_2,s_2).
\end{equation}
Denoting $\tilde s=qs_1+(1-q)s_2$ we get
\begin{equation}\label{QuoBiConEst}
	\frac{\Mijd\left(\tilde s\right)}{(\Mijd)^{**}\left(\tilde s\right)}\leq \frac{M\left(t_2,y_2,\tilde s\right)}{qM(t_1,y_1,\xi_1)+(1-q)M(t_2,y_2,s_2)}\leq\frac{qM(y_2,s_1)+(1-q)M(t_2,y_2,s_2)}{qM(t_1,y_1,\xi_1)+(1-q)M(t_2,y_2,\xi_2)}.
\end{equation}
Next, we observe that the definition of $\Mijd$ implies $M(t_1,y_1,s_1)=\Mijd(s_1)\leq M(t_2,y_2,s_1)$. We can assume without loss of generality that
\begin{equation}\label{MXiOneIneq}
	M(t_1,y_1,s_1)< M(t_2,y_2,s_1)
\end{equation}
because for $M(t_1,y_1,s_1)= M(t_2,y_2,s_1)$ inequality \eqref{QuoBiConEst} implies $\Mijd\leq(\Mijd)^{**}$ on $[s_1,s_2]$. Since we have always $\Mijd\geq(\Mijd)^{**}$ we arrive at $\Mijd=(\Mijd)^{**}$ on $[s_1,s_2]$.

Let us consider a function $h:[0,1]\rightarrow\R$ defined by
\begin{equation*}
	h(q)=\frac{qM(t_2,y_2,s_1)+(1-q)M(t_2,y_2,s_2)}{qM(t_1,y_1,s_1)+(1-q)M(t_2,y_2,s_2)}.
\end{equation*}
Then we compute
\begin{equation*}
	h'(q)=\frac{(M(t_2,y_2,s_1)-M(t_1,y_1,s_1))M(t_2,y_2,s_2)}{(q(M(t_1,y_1,s_1)-M(t_2,y_2,s_2))+M(t_2,y_2,s_2))^2}.
\end{equation*}
Obviously, we have $h'>0$ on $(0,1)$ due to \eqref{MXiOneIneq}. Therefore the maximum of $h$ is attained at $q=1$, which implies
\begin{equation}
	\frac{\Mijd\left(\tilde s\right)}{(\Mijd)^{**}\left(\tilde s\right)}\leq\frac{M(t_2,y_2,s_1)}{M(t_1,y_1,s_1)}.
\end{equation}
Then for sufficiently small $\delta$ we infer
\begin{equation}\label{SecQuoEst}
	\frac{\Mijd\left(\tilde s\right)}{(\Mijd)^{**}\left(\tilde s\right)}\leq\frac{M(t_2,y_2,s_1)}{M(t_1,y_1,s_1)}\leq \Theta^{iso}(|t_1-{t}_2|+c_x |x_1-{x}_2 | ,s_1)\leq \Theta^{iso}(|t_1-{t}_2|+c_x |x_1-{x}_2 |,s)\leq \Theta^{iso}(c\delta,s).
\end{equation}
since $y_1,y_2\in\tilde{Q}^\delta_j$ implies $|y_1-y_2|\leq 4\delta\sqrt{N}<1$. Combining \eqref{Quotient} with \eqref{FirQuoEst} and \eqref{SecQuoEst} we get
\begin{equation*}
\begin{split}	\frac{M(t,y,s)}{(\Mijd)^{**}(s)}\leq \Big(\Theta^{iso}(c\delta,s)\Big)^2=:\Theta (\delta,s). \end{split}
\end{equation*} 
Under the assumption $(\mathcal{M}_p)$, we shall have here $\Theta_p (\delta,s):=\Big(\Theta_p^{iso}(c\delta,s)\Big)^2$.
\end{proof}
 
 \subsection{Monotonicity trick}
 As a consequence of weak monotonicity, we are able to identify some limits using the following monotonicity trick applied e.g. in~\cite{pgisazg1,pgisazg2,gwiazda-ren-ell,Aneta}, but never extracted as a separate result.
\begin{lem}[Monotonicity trick]\label{lem:mon}
Suppose $A$ satisfies conditions {(A1)-(A3)} with an $N$-function $M$. Assume further that there exist \[{\cal A}\in L_{M^*}(\OT;\rn)\quad\text{ and }\quad \xi\in L_{M}(\OT;\rn),\] such that
\begin{equation}
\label{anty-mon}
\int_\OT \big({\cal A} -A(t,x,\eta)\big)\cdot(\xi(t,x) -\eta)\,dx\,dt\geq0\qquad\forall_{\eta\in L^\infty(\Omega,\rn)}.
\end{equation}
Then
\[A(t,x,\xi)={\cal A}\qquad\text{a.e. in }\ \OT.\]
\end{lem}
\begin{proof} Let us define\begin{equation*}
%\label{omK}
\Omega_T^K=\{(t,x)\in\OT:\ |\xi(t,x)|\leq K\}.
\end{equation*}
 Fix arbitrary $0<j<i$ and notice that $\Omega_T^j\subset\Omega_T^i$.

We consider~\eqref{anty-mon} with
\[\eta=\xi\mathds{1}_{\Omega_T^i}+hz\mathds{1}_{\Omega_T^j},\]
where $h>0$ and $z\in L^\infty(\OT;\rn)$, namely
\begin{equation*}
%\label{Ak-po-mon}
\iOT( {\cal A}- A(t,x,\xi\mathds{1}_{\Omega_T^i}+hz\mathds{1}_{\Omega_T^j}) )\cdot( \xi-\xi\mathds{1}_{\Omega_T^i}-hz\mathds{1}_{\Omega_T^j})\, dx\,dt\geq  0.
\end{equation*}
Notice that it is equivalent to
\begin{equation}
\label{A-po-mon}\begin{split}& \int_{\OT\setminus\Omega_T^i}  \big({\cal A} -A(t,x,0)\big)\cdot\xi\,dx\,dt  +h\int_{ \Omega_T^j} (A(t,x,\xi+hz)-{\cal A} )\cdot z\,  dx\,dt \geq  0.\end{split}
\end{equation}
The first expression above tend to zero when $i\to\infty.$ Indeed,  { (A2)} implies $A(t,x,0)=0$, moreover ${\cal A}\in L_{M^*}(\OT;\rn)=\big(E_{M}(\OT;\rn)\big)^*$ and $\xi\in L_M(\OT;\rn)$, the H\"older inequality~\eqref{inq:Holder} gives boudedness of integrands in $L^1(\OT)$. Then we take into account shrinking domain of integration to get the desired convergence to $0$. In~particular, we can erase these expressions in~\eqref{A-po-mon} and divide the remaining expression by $h>0$, to obtain
\begin{equation*} \int_{ \Omega_T^j} (A(t,x,\xi+hz)-{\cal A} )\cdot z\,  dx\,dt\geq  0.
\end{equation*}

Note that
\[A(t,x,\xi+hz)\xrightarrow[h\to 0]{}A(t,x,\xi)\quad\text{a.e. in}\quad \Omega_T^j.\]
Moreover, as $A(t,x,\xi+hz)$ is bounded on $\Omega_T^j$, (A2) results in
\[c_A\int_{ \Omega_T^j} M^*\left(t,x,A(t,x,\xi+hz)\right)\,dx\,dt\leq   \sup_{h\in(0,1)} \int_{ \Omega_T^j} M\left(t,x, \left( \xi+hz\right) \right)\,dx\,dt.\]
The right-hand side is bounded, because $(\xi+hz)_h$ is uniformly bounded in $L^\infty(\Omega_T^j;\rn)\subset L_M(\Omega_T^j;\rn)$. Indeed, on $\Omega_T^j$ by definition $|\xi|\leq j$. Hence, Lemma~\ref{lem:unif}  gives uniform integrability of $\left(A(t,x,\xi+hz)\right)_h$. When we notice that $|\Omega_T^j|<\infty$, we can apply the Vitali Convergence Theorem (Theorem~\ref{theo:VitConv}) to get
\[A(t,x,\xi+hz)\xrightarrow[h\to 0]{}A(t,x,\xi)\quad\text{in}\quad L^1(\Omega_T^j;\rn).\]
Thus
\begin{equation*} \int_{ \Omega_T^j} (A(t,x,\xi+hz)-{\cal A} )\cdot z \, dx\,dt\xrightarrow[h\to 0]{} \int_{ \Omega_T^j} (A(t,x,\xi)-{\cal A} )\cdot z \, dx\,dt.
\end{equation*}
Consequently,
\begin{equation*}  \int_{ \Omega_T^j} (A(t,x,\xi)-{\cal A} )\cdot z\,  dx\,dt\geq 0,
\end{equation*}
for any $z\in L^\infty(\OT;\rn)$. Let us take
\[z=\left\{\begin{array}{ll}-\frac{A(t,x,\xi)-{\cal A} }{|A(t,x,\xi)-{\cal A} |}&\ \text{if}\quad A(t,x,\xi)-{\cal A} \neq 0,\\
0&\ \text{if}\quad A(t,x,\xi)-{\cal A} \neq 0.
\end{array}\right.\]
We obtain
\begin{equation*}  \int_{ \Omega_T^j} |A(t,x,\xi)-{\cal A} |\, dx\,dt\leq 0,
\end{equation*}
hence \[A(t,x,\xi)={\cal A} \quad \text{a.e.}\quad\text{in}\quad \Omega_T^j.\]
Since $j$ is arbitrary, we have the equality a.e. in $\OT$ and~\eqref{lim=ca} is satisfied.\end{proof}

\section{References}
\bibliographystyle{plain}
\bibliography{arXiv-Para-t}
\end{document}